%% file: AIHP838.tex
\documentclass[12pt,letterpaper]{amsart}
\usepackage{mathabx}
\usepackage{nameref}
\usepackage{graphicx}
\usepackage{xcolor}
\usepackage{caption}
\usepackage{subcaption}
\usepackage{systeme}
\usepackage{float}
\usepackage[T1]{fontenc}
\usepackage{amsmath}
\usepackage{mathtools}
\usepackage{epic,eepic,latexsym, amssymb, amscd, amsfonts}
\usepackage[all]{xy}
\usepackage{url}
\usepackage{hyperref}
\usepackage{enumerate}
 \newlength{\baseunit}               
 \newcount{\numlines}                
\input{defs1.tex}

\begin{document}
\pagestyle{plain}
\title{Asymptotics of random domino tilings of rectangular Aztec diamonds}
\author{Alexey Bufetov}

\address[Alexey Bufetov]{Department of Mathematics, Massachusetts Institute of Technology, Cambridge, MA, USA. E-mail: alexey.bufetov@gmail.com}

\author{Alisa Knizel}

\address[Alisa Knizel]{Department of Mathematics, Massachusetts Institute of Technology, Cambridge, MA, USA. E-mail: alisik@math.mit.edu}

\begin{abstract}
\input{abstract.tex}
\end{abstract}
\maketitle
\input{intro1.tex}

\input{combinatorics.tex}

\input{prob.tex}
\input{probheight.tex}

\input{gen_measure_new.tex}
\input{coordinates.tex}
\input{hom.tex}

\input{fluctuations.tex}
\input{examples.tex}

\input{appendixA.tex}
\input{AppendixB.tex}

\input{refs.tex}
\end{document}

%% file: defs1.tex
\makeatletter
\newcommand*{\rom}[1]{\expandafter\@slowromancap\romannumeral #1@}
\makeatother

\numberwithin{equation}{section}

\newtheorem{thm}{Theorem}
\numberwithin{thm}{section}
\newtheorem{lem}[thm]{Lemma}
\newtheorem{Con}[thm]{Construction}
\newtheorem{defi}[thm]{Definition}
\newtheorem{prop}[thm]{Proposition}
\newtheorem{cor}[thm]{Corollary}
\newtheorem{rem}[thm]{Remark}

\newcommand{\pa}{\partial}
\newcommand{\mes}{{\boldsymbol{\eta}}}
\newcommand{\F}{ \mathbf F}
\newcommand{\ka}{\kappa}

\newcommand{\w}{\omega}

\newcommand{\ii}{{\mathbf i}}

\newcommand{\GT}{\mathbb{GT}}
\newcommand{\la}{\lambda}
\newcommand{\st}{\mathrm{st}}
\newcommand{\pr}{\mathrm{pr}}
\newcommand{\E}{\mathbf E}
\newcommand{\cov}{\mathrm{cov}}
\newcommand{\Cov}{\mathrm{Cov}}
\newcommand{\ep}{\varepsilon}

\def\L{\textup{L}} 
\def\La{\mathcal L}

\def\d{\textup{det}}

\def\H{\textup{H}}
\def\a{\Omega}
\def\b{\beta}

\def\w{\omega}

%% file: abstract.tex
We consider asymptotics of a domino tiling model on a class of domains which we call rectangular Aztec diamonds.  We prove the Law of Large Numbers for the corresponding height functions and provide explicit formulas for the limit. For a special class of examples, the explicit parametrization of the frozen boundary is given. It turns out to be an algebraic curve with very special properties. Moreover, we establish the convergence of the fluctuations of the height functions to the Gaussian Free Field in appropriate coordinates. Our main tool is a recently developed moment method for discrete particle systems. 

%% file: intro1.tex
\section{Introduction}
\label{sec:1}
We study the asymptotic behavior of uniformly random domino tilings of domains drawn on the square grid. This model has received a significant attention in the last twenty five years (\cite{CEP}, \cite{CKP}, \cite{EKLP}, \cite{JPS}, \cite{J2}, \cite{K}). Let us briefly describe our results.

We consider a class of domains, which we call {\it rectangular Aztec diamonds}, see Figure \ref{fig:rect} for an example. This type of domains generalizes a well-known case of the Aztec diamond, introduced in \cite{EKLP}, at the same time inheriting many of its combinatorial properties. For instance, similar to the Aztec diamond, the domains we consider also have a rectangular shape with sawtooth boundary.

 \begin{figure}[h]
\includegraphics[width=0.45\linewidth]{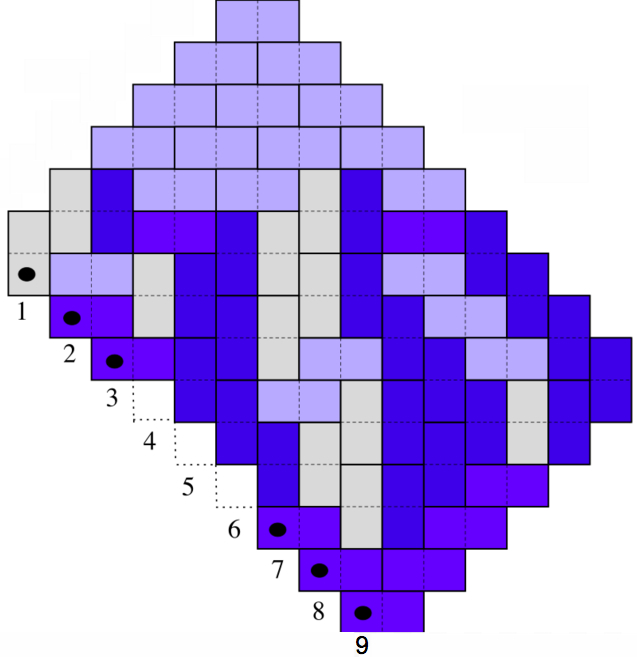}
\caption{Domino tiling of a Rectangular Aztec diamond $\mathcal R(6,
  (1, 2, 3, 7, 8, 9), 3 )$.}
  \label{fig:rect}
\end{figure}

The main feature of this class of domains is a variety of different boundary conditions which are allowed on one side of the rectangle. These boundary conditions are parameterized by configurations of boxes with dots as presented in Figure \ref{fig:rect}. When the mesh size goes to zero the limit behavior of the boundary boxes can be parameterized by a probability measure $\mes$ on $\mathbb R$ with a compact support.
We are able to analyze a global asymptotic behavior of the uniform random tiling of such a domain for an arbitrary choice of this measure.


\textbf{Limit shape.} A domino tiling can be conveniently parameterized by the so-called \textit{height function}, see Section \ref{sec:height}.  It is an integer-valued function on the
vertices of the square grid inside the  domain, which satisfies certain
conditions (see Definition \ref{defi_height} for further details). There is a one-to-one
correspondence between tilings and height functions (as long as
the height is fixed at one vertex). A random domino tiling naturally gives rise to a random height function.
In Theorem \ref{LLN} we prove that for an arbitrary measure $\mes$ a random height function converges to a deterministic function as the mesh size of the grid goes to zero, furthermore, we give an explicit formula for it.

\begin{figure}[h]
\includegraphics[width=0.4 \linewidth]{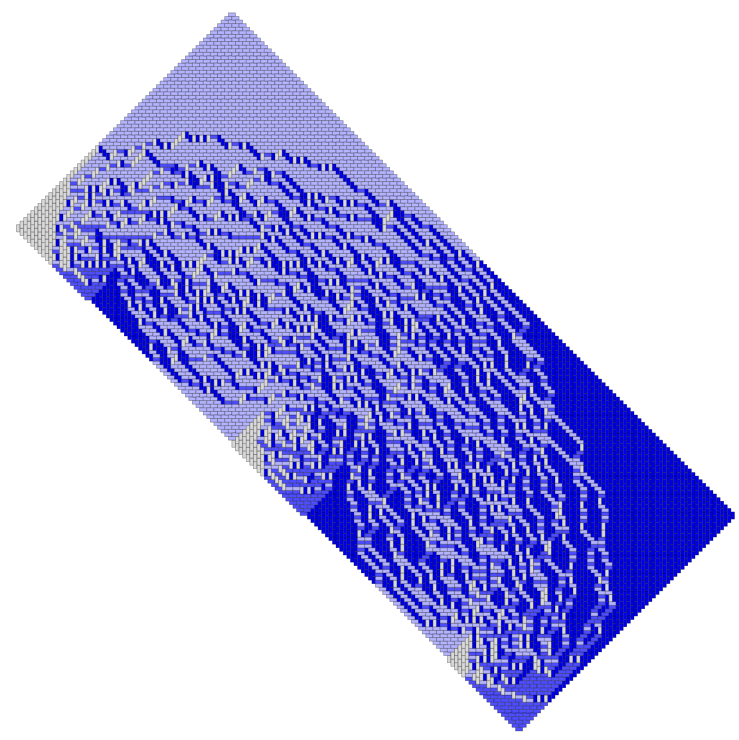}
\caption{A limit shape simulation. There is a formation of
  brick-wall pattern in the areas near the boundary of the domain. }
  \label{fig:example}
\end{figure}

In \cite{CKP} it was shown that the limit shape exists for a wide class of domains on the square grid and can be found as a solution to a certain variational problem. Our approach is different and provides an explicit formula for the limit height function. Our computation of the limit shape is closely related to the notion of the free projection from the Free Probability Theory (see Remark \ref{rem:lln}).

\textbf{Frozen boundary.} Typically, a limit shape forms frozen facets, that is areas, where only one type of domino is present. There also exists a connected
open {\it liquid region} inside the domain in which arbitrary local configurations of dominos
are present. The curve, which separates the
liquid region from the frozen zones is called {\it frozen
  boundary}.

In a particular case, when the measure $\mes$ is a uniform measure of density $1$ on a union of $s$ segments such that their lengths add to $1$, we give an explicit parametrization of the frozen boundary, see Theorem \ref{frozen_b}. It turns out to be an algebraic curve of rank $2s$ and genus zero with very special properties. The degree of the frozen boundary linearly depends on the number of segments. Therefore, the subclass of domains we consider provides a diverse variety of limit shapes of an arbitrary complexity.

\begin{figure}[h]
\includegraphics[width=0.4\linewidth]{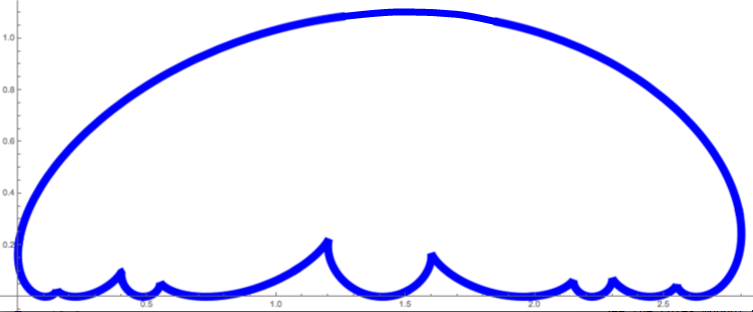}
\caption{An example of the frozen boundary with $s=5.$ }
  \label{fig:boo}
\end{figure}


Moreover, our formulas allow to analyze the frozen boundary for an arbitrary measure $\mes$. We discuss several examples in Section \ref{sec:examples}.

\textbf{Fluctuations.} For arbitrary boundary conditions on one side of the rectangular Aztec diamond we prove a Central Limit Theorem for a global behavior of the random height function (see Theorem \ref{theorem:covariance-domino}), which is the main result of this paper. We show that after a suitable change of coordinates the fluctuations are described by the \textit{Gaussian Free Field}. The appearance of the Gaussian Free Field as a universal object in this class of probabilistic tiling models originates from the work of Kenyon (see \cite{K}, \cite{K1} ).

In \cite{K} Kenyon proved a central limit theorem for uniformly random tilings of domains of an arbitrary shape, but with very special boundary conditions such that the limit shape does not have any frozen facets. In contrast, we analyze rectangular domains with arbitrary boundary conditions on one of the sides and the limit shape in our case always has frozen facets.  The 
fluctuations of the liquid
region for a random tiling model containing both frozen facets and liquid
region were first studied in \cite{BF1}.

Depending on the boundary conditions, the Law of Large Numbers can have a quite complicated form. It is reflected in a (possibly complicated) choice of the \textit{complex structure}, that is a map from the liquid region into the complex half-plane. In other words, it is the choice of the coordinate in which the Gaussian Free Field appears as a limit object.

\bigskip

A parallel (and actually more developed) story exists for the case of \textit{lozenge} tilings. We refer to \cite{K2} for the exposition and further references on the subject. Both the domains we consider and the fluctuation results are close in spirit to \cite{P1}, \cite{P2}, however, the approach we take is entirely different.

 We use a moment method for this type of problems. It was introduced and developed in \cite{BG}, \cite{BG2}. Let us comment on two other known methods. The method based on the study of a family of orthogonal polynomials, which was extensively used in the case of the Aztec diamond, does not seem to be available in our setting. A large class of tiling problems fits the framework of Schur processes which was introduced in \cite{OR}. It was shown in \cite{OR} that any Schur process is a determinantal process with a correlation kernel suitable for asymptotics analysis. Papers \cite{DM},\cite{P1}, \cite{P2} study the lozenge tiling model which is combinatorially similar to a Schur process yet does not fit this framework; a significant effort was necessary to derive a correlation kernel there. It is an important challenge to find a reasonable correlation kernel for the tiling model studied in the present paper and to perform its asymptotic analysis. However, we believe that the moment method
is the most suitable method for the analysis of the global behavior in this class of problems (see Remark \ref{rem:clt} for further comments).

In this paper we give a ``model case'' analysis of a specific class of domino tilings. However, we believe that the moment method and the tools developed in this paper are applicable in many other models. Let us mention some of them.

The domino tiling model considered in this paper can be interpreted as a random collection of non-intersecting lines, see Figure \ref{fig:lines} and Section \ref{sec:paths} for details. In the case of the Aztec diamond this interpretation was first used by Johansson in \cite{J}, and later many similar and more general ensembles of non-intersecting lines were studied, see \cite{BS}, \cite{BF}, \cite{BCC}, \cite{BBCCR} and references therein. Some of these ensembles can be viewed as dimer models. More precisely, in \cite{BBCCR} it was shown how to interpret an arbitrary Schur process as a dimer model on a so-called rail yard graph.

The main novelty of the presented approach is that we consider non-intersecting lines with \textit{arbitrary} boundary conditions on one side. We suggest that all such models can be analyzed with the use of the moment method and the results of this paper. Because the global behavior of the height function significantly depends on the boundary conditions in all these models, we expect further interesting results in this direction.

This paper is organized as follows. In Section \ref{sec:2} we discuss combinatorial properties of rectangular Aztec diamonds. In Sections \ref{sec:3} and \ref{sec:4} we analyze the limit shape of the tilings. In Section \ref{sec:frozen} we study the frozen boundary for a specific class of examples. In Section \ref{sec:6} we prove the global Central Limit Theorem. In Section \ref{sec:examples} we study some examples which are not covered by Section \ref{sec:frozen}. In \nameref{sec:A} we briefly comment on a more general class of probability measures on rectangular Aztec diamonds. In \nameref{sec:B} we provide a result on the local behavior of these tilings.

\begin{figure}[h]
\includegraphics[width=0.3 \linewidth]{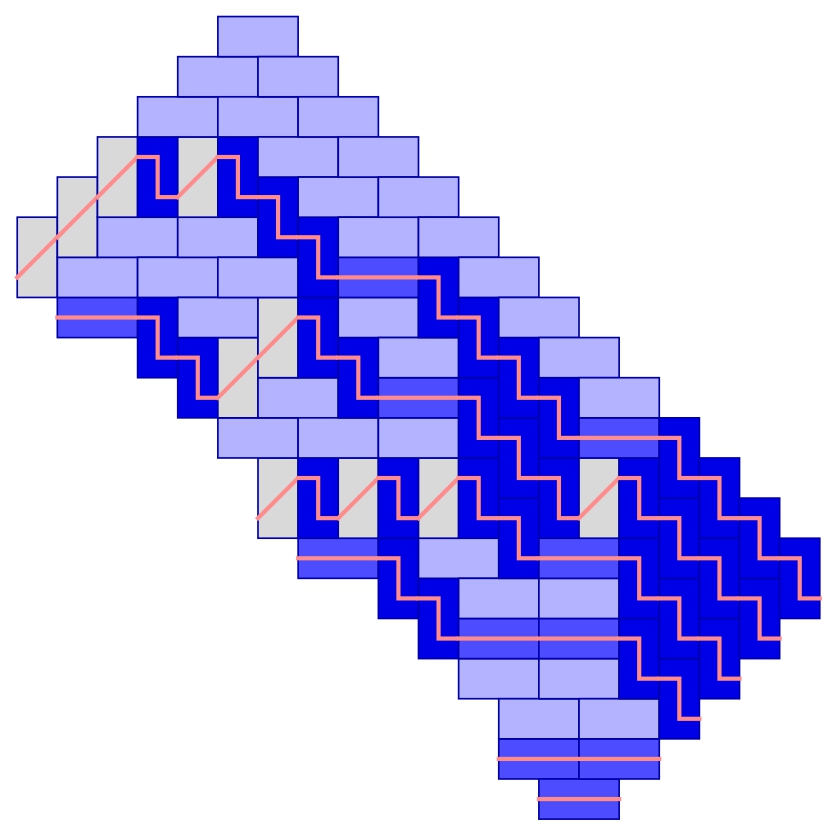}
\caption{Rectangular Aztec diamond $\mathcal R(6, (1,2,4,5,7,8), 12)$ and the corresponding set of non-intersecting lines. }
  \label{fig:lines}
\end{figure}

\subsection*{Acknowledgements}  The authors are deeply grateful to Alexei Borodin and Vadim Gorin for very helpful discussions. The authors would like to thank anonymous referees for many valuable comments which helped to improve the manuscript.

%% file: combinatorics.tex
 \section{The model description} 
\label{sec:2}
In this section we study the combinatorics of the model. We establish
a bijection between the domino tilings of a rectangular Aztec diamond
and the sequences of Young diagrams with some special properties. This will
allow us to bring in the machinery of the Schur functions to study the
asymptotics. The key observation is Proposition \ref{uniform}.

In the end of the section we discuss another combinatorial realization
for our model through the non-intersecting line ensembles, which was mentioned in
introduction.

\subsection{Combinatorics of the model}
Let us present a formal definition of the domain we are considering.

\begin{defi}
Let $N\in \mathbb N$ and $\a=(\a_1, \dots, \a_N)$, where $1= \a_1
<\a_2<\dots<\a_N$ and $\a_i\in\mathbb N$. Set $m=\a_N-N.$

Introduce the coordinates $(i, j)$ as in Figure ~\ref{example}.
Let us denote by $C(i,j)$ a unit square with the vertex coordinates $(i, j),$ $(i-\frac{1}{2}, j+\frac{1}{2}),$ $(i+\frac{1}{2}, j+\frac{1}{2}),$ $(i, j+1).$

The rectangular Aztec diamond $\mathcal R(N, \a, m)$ is a polygonal
domain which is built of $2N+1$ rows of unit squares $C(i, j).$ Let us enumerate the rows starting from the bottom as shown in Figure ~\ref{example}. Then
the $k$-th row consists of unit squares $C(i, j),$ where
\begin{itemize}
 \item $i=0,1,\dots, \a_N$ and $j=0,1,\dots, N-1,$ when $k=2s,$ $s=0,1,\dots, N;$
\item $i=\frac{1}{2}, \frac{3}{2}, \dots, \a_N-\frac{1}{2}$ and $j=\frac{1}{2},\frac{3}{2},\dots, N-\frac{1}{2},$ when $k=2s+1,$ $s=1,2,\dots, N;$
\item $i=\a_l-\frac{1}{2}$ and $j=-\frac{1}{2}$ for $l=1,2,\dots , N,$ when $k=1.$  We
call this row {\it the boundary} of $\mathcal R(N, \a, m)$.
\end{itemize}
\end{defi} 

 \begin{figure}[h]
\includegraphics[width=0.35\linewidth]{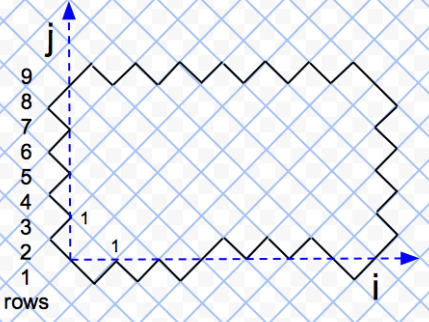}
\caption{Rectangular Aztec diamond $\mathcal R(4, \a=(1,2,3, 7), 3).$}
  \label{example}
\end{figure}

\begin{defi}
A domino tiling of a rectangular Aztec diamond $\mathcal R(N, \a,m)$ is a
set of pairs $(C(i_1, j_1), C(i_2, j_2)),$ called dominos, such that the unit squares
$C(i_1,j_1), C(i_2, j_2)\subset \mathcal R(N, \a,m)$ share an edge and
every unit square belongs to exactly one domino.

Let us denote the set of domino tilings of  $\mathcal R(N, \a,m)$ by
$\mathfrak D(N,\a, m).$
\end{defi}

\begin{defi}
Let $D\in\mathfrak D(N, \a, m )$ be a domino tiling of $\mathcal R(N,
\a, m).$ We call a domino $d=(C(i_1, j_1), C(i_2, j_2))\in D$ a $V$-domino
if $max (i_1, i_2)\in \mathbb N$ and we call it a $\Lambda$-domino otherwise. In other words, the dominos going upwards starting from an odd row are $V$-dominos and those ones starting from an even row are $\Lambda$-dominos. We also call the corresponding
squares  $C(i_1, j_1), C(i_2, j_2)$ --- $V$-squares and
$\Lambda$-squares accordingly.
\end{defi}

     \begin{figure}[h]
\includegraphics[width=0.35\linewidth]{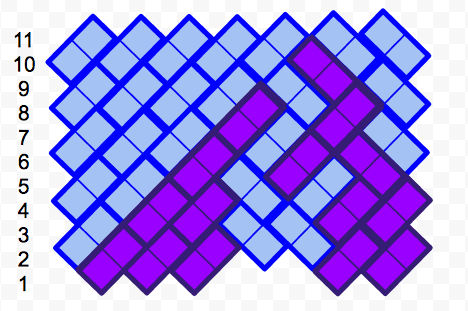}
\caption{Domino tiling of $\mathcal R(4, \a=(1,2,3, 7),3).$
  The $V$-squares are purple and $\Lambda$-squares are blue. }
  \label{fig:vdomino}
\end{figure}

\begin{lem} \label{good}
 The set of V-squares in $\mathcal R(N,
\a, m)$ determines the tiling uniquely.
\end{lem}
\begin{proof}
Let us reconstruct the tiling given the set of V-squares coming
from some unknown tiling. We will show that we can reconstruct it in a
unique way. Let us enumerate the squares in each row from left to right.

We start by looking at the V-squares in the first row. All squares in
the first row are V-squares and there are exactly $N$ of them. Let us take the first V-square in the first row. We will pair it with the first V-square in the second row. We can always do it since we know that there exists at least one domino tiling with this set of V-squares. We can proceed until the end of the first row. Note that at this point we have used all V-squares from the second row and this is the only way we could pair the V-squares from the first row with something.

Now we pair the first $\Lambda$-square in the second row with the first $\Lambda$-square in the third row. We can proceed until the end of the second row. Note that at this point we have used all squares from the second row and all  $\Lambda$-squares from the third row.  Moreover, this is the only way we could pair $\Lambda$-squares from the second row with something. Now we look at V-squares in the third row and notice that there are $N-1$ of them. This is because the total number of squares in the third row is less by one than the total number of squares in the second row. We start pairing them with V-squares from the third row in the same way.

We continue in this fashion until we pair all $\Lambda$-squares from row $2N$ with all the squares from row $2N+1.$

\end{proof}

\begin{defi} Let $\mu$ be a $n$-tuple of natural numbers $\mu=(\mu_1
  \geq \mu_2 \dots \geq \mu_n\geq 0).$  We will denote $\ell(\mu)=n.$ A Young diagram $Y_{\mu}$ is a set of boxes in
  the plane with $\mu_1$ boxes in the first row, $\mu_2$
  boxes in the second row, etc.

We call a Young diagram $Y_\mu$ rectangular $n\times m$ Young diagram if
$m=\mu_1=\mu_2=\dots=\mu_n.$ Let us denote it $Y_{n\times m}.$
\end{defi}

\begin{defi}
Let $Y_{\mu}$ be a Young diagram, where $\mu=\mu_1 \geq \mu_2 \dots
\geq \mu_n\geq 0.$ A dual Young diagram $Y_{\mu^{\vee}}$ is obtained by
taking the transpose of the original diagram $Y_{\mu}$. Explicitly
we have that $\mu^{\vee}_i$ is equal to the length of the $i$-th column of $Y_{\mu}.$

\end{defi}

\begin{Con} \label{con}  
 Due to Lemma \ref{good} we can encode any domino tiling
$D\in\mathfrak D(N, \a, m)$ pictorially as it is shown in Figure
$\ref{fig:points}.$ We simply put a yellow node in the center of every V-square from an even row and a pink node in the center of every V-square from an odd row. The configuration of these nodes determines the tiling uniquely. In the case of the Aztec diamond such encoding was suggested by Johansson \cite{J}.

 \begin{figure}[h]
\includegraphics[width=0.35\linewidth]{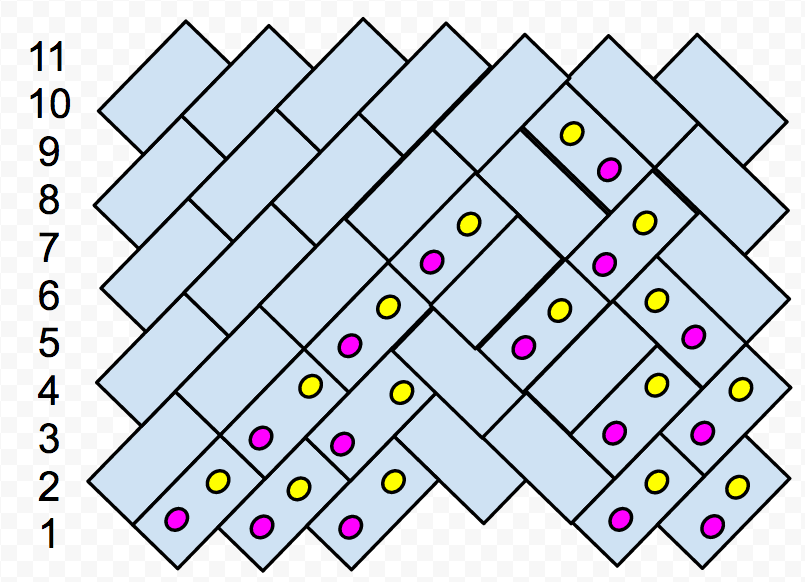}
\caption{An example of a domino tiling of  $\mathcal R(5, \a=(1,2,3,
  6,7), 3).$ We put nodes in the centers of $V$-squares. }
  \label{fig:points}
\end{figure}

Let us associate to every configuration of the nodes the following
sequence of Young diagrams $Y_j$. Each diagram $Y_j$
corresponds to the $j$-th row $j=1, 2, \dots ,2N$, which has $n_j$
V-squares and $m_j$ $\Lambda$-squares. Let us take a rectangular
$n_j\times m_j$ Young diagram. We start drawing a stepped line $\ell_j$ inside it
starting from the left bottom corner. At step $k$ if
the $i$-th square in the row is a V-square $\ell_j$ goes up by one, otherwise
the line goes to the right by
one. The stepped line $\ell_j$ is the boundary of
a Young diagram, see Figure $\ref{fig:diagram}.$ 

Let us look at the
$j$-th row. It has $N-\lfloor \frac{j-1}{2} \rfloor$ V-squares in positions
$i_1, \dots, i_{N-\lfloor \frac{j-1}{2} \rfloor}$ starting from the
left. Then the Young diagram we obtain corresponds to the $N-\lfloor
\frac{j-1}{2} \rfloor$-tuple $(i_{N-\lfloor \frac{j-1}{2} \rfloor}-\lfloor
\frac{j-1}{2} \rfloor, \dots, i_2-2,
  i_1-1).$

To construct a Young
diagram corresponding to the boundary row we complete the row by
virtually adding $m$ $\Lambda$-squares so that $\mathcal R(N, \a,m)$ becomes
a proper rectangular with sawtooth boundary. Let
us denote the Young diagram corresponding to the boundary row
$Y_{\omega},$ where $\omega$ is an $N$-tuple of integers, more precisely,  $\omega=(\a_N-N, \dots, \a_1-1). $
\end{Con}

 \begin{figure}[h]
\includegraphics[width=0.45\linewidth]{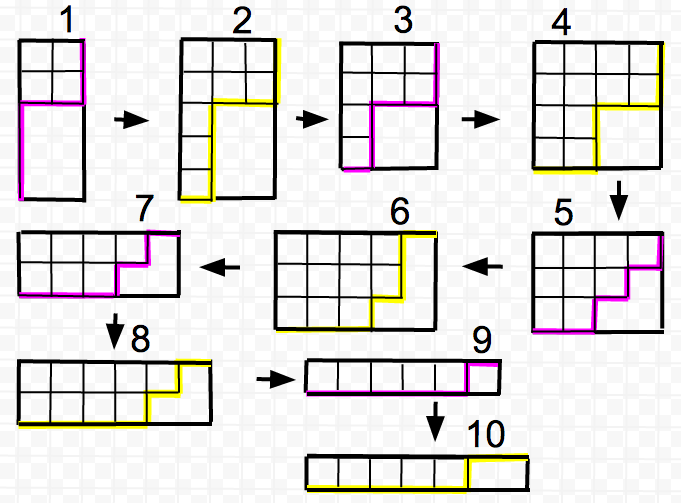}
\caption{A sequence of Young diagrams corresponding to the domino
  tiling $\mathcal R(5, \a=(1,2,3, 6,7),3)$ presented in Figure
  $\ref{fig:points}.$}
  \label{fig:diagram}
\end{figure}

\begin{defi}
Let $Y_{\omega}$, where $\omega=\w_1 \geq \w_2 \dots \geq \w_N \geq 0,$ be a Young diagram contained in a rectangle $N\times m$ Young diagram $Y_{N\times m}$. Let $\mathcal S(Y, N, m)$ be the following set of sequences of Young diagrams

$$\mathcal S(Y_{\mu}, N, m)=\{(Y_{\mu^{(N)}}=Y_\omega, Y_{\nu^{(N)}},\dots,Y_{\mu^{(1)}}, Y_{\nu^{(1)}}\},$$

such that
\begin{itemize} \label{prop} 
\item $\ell(\mu^{(i)})=i$ and $\ell(\nu^{(i)})=i$ for $i=1, 2,\dots, N.$ 

\item $ Y_{\mu^{(i)}}\subset  Y_{i\times (m+N-i)}$ for $i=1, 2, \dots,
  N;$ 
\item $ Y_{\nu^{(i)}}\subset Y_{i\times (m+N-i+1 )}$  for $i=1, 2, \dots,
  N;$
\item  $Y_{\mu^{(i)}}\subset Y_{\nu^{(i)}}$ for $i=1, 2, \dots, N;$
\item $Y_{\nu^{(i)}}$ $\setminus$ $Y_{\mu^{(i)}}$ is a vertical strip of
  length $l\leq i$ for $i=1, 2,\dots, N;$ that is $ Y_{\nu^{(i)}}$ can be obtained from $
 Y_{\mu^{(i)}}$ by adding $l$ boxes, no two in the same row, see Figure
  \ref{strip}.
\item $Y_{\nu^{(i+1)}}\setminus Y_{\mu^{(i)}}$ is a horizontal
  strip. In other words, they interlace $(Y_{\mu^{(i)}}\prec Y_{\nu^{(i+1)}}),$ that is
$$\nu^{(i+1)}_1\geq \mu^{(i)}_1\geq \dots \geq
\nu^{(i+1)}_{i}\geq \mu^{(i)}_{i}\geq \nu^{(i+1)}_{i+1},  \text{ for } i=1, 2, \dots, N-1.$$
\end{itemize}

\end{defi}

 \begin{figure}[h]
\includegraphics[width=0.30\linewidth]{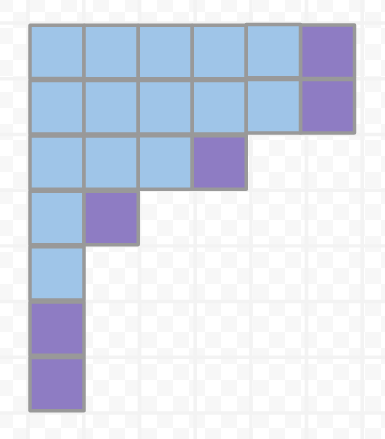}
\caption{Purple boxes form a vertical strip of length $6.$}
  \label{strip}
\end{figure}

\begin{thm}\label{bij}
Construction \ref{con} defines a map $$Y\colon \mathfrak
D(N,\a, m)\rightarrow \mathcal
S(Y_{\omega}, N, m).$$ The map is bijective.
\end{thm}
\begin{proof}
Let $\mathcal D\in
\mathfrak D(N,\a, m)$ be a domino tiling and let $Y(\mathcal
D)=(Y_{\mu^{(n)}=\omega}, Y_{\nu^{(N)}},\dots,Y_{\mu^{(1)}}, Y_{\nu^{(1)}})$. Let us first check that this map
is well-defined, i.e. $Y(\mathcal
D) \in  \mathcal
S(Y_\omega, N, m).$ Thus, we need to verify all the properties in Definition
\ref{prop}. Figure \ref{fig:step} illustrates the proof.

 By construction $Y_{\mu^{(i)}} \subset Y_{n_i\times m_i},$ where $n_i$ and
  $m_i$ are the number of $V$-squares and $\Lambda$-squares
 in the row with number $2(N-i)+1$. Following the proof of Lemma
  \ref{good} we see that $n_i=i$ and $m_i=N+m-i.$ Also, the length of
  $Y_{\mu^{(i)}}$ is equal to the number of $V$-squares in the
  corresponding row. Thus, we get $\ell(\mu^{(i)})=i.$
 Similarly we have  
$Y_{\nu^{(i)}}\subset Y_{i\times (m+N-i+1 )}$ and $\ell(\nu^{(i)})=i.$

 \begin{figure}[h]
\includegraphics[width=0.4\linewidth]{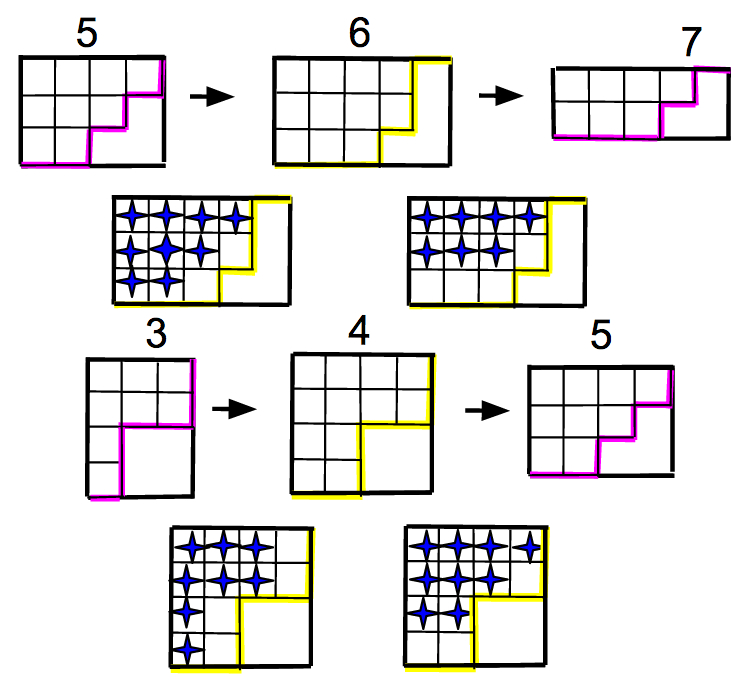}
\caption{ }
  \label{fig:step}
\end{figure}

Note that by construction $\mu_j^{(i)}$ is precisely the number of
$\Lambda$-squares to the left from the $j$-th $V$-square in the
corresponding row with number $k=2(N-i)+1$ . Let us look at this row. The
  $V$-squares from it are paired with $V$-squares from the next
  row in $\mathcal D$. It can be shown by induction that the number of
  $\Lambda$-squares to the left from the $j$-th $V$-square
 in the $k+1$ row is always equal or greater by one than the same quantity for the $j$-th $V$-square in the $k$-th row. It follows from
  the fact that if the $j$-th $V$-square is in position $n$ in the $k$-th row
  (i.e. it is the $n$-th square from the left in this row) it must be
  paired with the $j$-th $V$-square from the next row in position $n$
  or $n+1,$ see Figure \ref{fig:vdomino}. Therefore, $Y_{\mu_j^{(i)}}\subset
  Y_{\nu_j^{(i)}}$ and $Y_{\nu^{(i)}}$ $\setminus$ $Y_{\mu^{(i)}}$ is a vertical strip with a number of
  blocks $\leq i.$

 Consider $Y_{\nu^{\vee, (i)}},$ the Young diagram dual to $Y_{\nu^{(i)}}.$
   Note that by construction $\nu^{\vee, (i)}_j$ is precisely the number of
$V$-squares to the left from the $j$-th $\Lambda$-square in the
corresponding row with number $k=2(N-i)+2$, where this time we count
from the right. Let us look at this row. Then the $\Lambda$-squares
from the next row are paired with the $\Lambda$-squares from this row,
moreover, the $j$-th $\Lambda$-square in position $n$ is paired with
the $j$-th $\Lambda$-square in position $n$ or $n+1$ from the next row
counting from the right. Again, we
get that $Y_{\nu^{\vee, (i)}}$ $\setminus$ $Y_{\mu^{\vee, (i+1)}}$ is a vertical strip. Thus, $Y_{\mu^{(i+1)}}$ $\setminus$ $Y_{\nu^{(i)}}$ is a horizontal strip. Therefore, the
 interlacing condition holds.

So far we have checked that the map $Y$ is well-defined. From Lemma
\ref{good} it follows that it is injective.

Let us construct the inverse map $Y^{(-1)} \colon \mathcal
S(Y_{\alpha}, N, m) \rightarrow \mathfrak
D(N,\a, m).$ Inverting the Construction \ref{con} we see that each
element $y\in\mathcal
S(Y_{\alpha}, N, m)$ defines a configuration of V-squares. Our goal is
to show that there exists a tiling with such a configuration of
V-squares. We can reconstruct it using the same ideas as we used in
the proof of Lemma
\ref{good}. We start by pairing V-squares from the first row with
V-squares from the second row starting from the left. We can always
pair the $j$-th V-square from the first row in position $n$ with the $i$-th V-square from the second row
because it has to be in position $n$ or $n+1$ due to our assumptions
on $y.$ Then we proceed to the
next row and pair $\Lambda$-squares with $\Lambda$ squares in the
third row. The map $Y^{(-1)}$ is then well-defined and
injective. Therefore, $Y$ is a bijection.
\end{proof}

%% file: prob.tex
\subsection{Uniform measure on $\mathcal S(Y_{\alpha}, N, m)$.}
\label{sec:22}
Theorem \ref{bij} allows us to reduce any question about the uniform
measure on the set of domino tilings $\mathfrak D(N, \a, m)$ to the
same question for the uniform measure on $\mathcal S(Y_{\alpha}, N, m).$ This is the core of
our approach.

Let $U(N)$ denote the group of all $N\times N$ complex
unitary matrices. It is well-known that all irreducible
representations of $U(N)$ are parameterized by their highest weights,
which are \textit{signatures} of length $N,$ that is, $N$-tuples of integers
$\lambda=(\lambda_1\geq \lambda_2\geq\dots \geq \lambda_N).$ We denote by
$\GT_N$ the set of all signatures and the length of a signature is denoted by $\ell(\lambda)$. We call a signature
\textit{non-negative} if $\lambda_N\geq 0.$ Note that the set of all nonnegative
signatures $\GT_N^{+}$ is in bijection with the set of Young diagrams
$\mathbb Y_N$ with $N$ rows (rows are allowed to have zero length).

\begin{defi}
 Let $\lambda\in \GT_N.$ The rational Schur function is
\begin{equation}
s_{\lambda}(u_1,\dots,
u_N)=\frac{\d_{i,j=1,\dots,N}(u^{\lambda_j+N-j}_i)}{\prod\limits_{1\leq
    i<j\leq
    N}(u_i-u_j)}.
\end{equation}
\end{defi}

 \begin{prop}$ ($Weyl, \cite{W}$)$ The value of the character of the irreducible representation $\pi^\lambda$ corresponding to the signature
 $\lambda=(\lambda_1\geq \dots \geq \lambda_N )$ on a unitary matrix $\mathfrak u\in U(N)$ with eigenvalues $u_1,\dots,u_n$ is given by the rational Schur function:
$$ \textup{Trace}(\pi^\lambda(\mathfrak u))=s_\lambda(u_1,\dots,u_N). $$

 \end{prop}

Let $\mu^{(n)}$ and $\nu^{(n)}$ be two non-negative signatures of length
$n.$ Recall that Schur functions form a basis for the algebra of symmetric functions.
Define the coefficients
$\st(\mu^{(n)} \rightarrow \nu^{(n)})$ and $\pr(\nu^{(n)} \rightarrow \mu^{(n-1)})$ via

\begin{equation}  \label{s}
\frac{s_{\mu^{(n)}}(u_1,\dots,u_n)}{s_{\mu^{(n)}}(1^n)} \displaystyle \prod \limits^{n}_{i=1}\frac{(1+u_i)}{2}=\sum_{\nu^{(n)}\in\GT_n}\st(\mu^{(n)} \rightarrow \nu^{(n)}) \frac{s_{\nu^{(n)}}(u_1,\dots,u_n)}{s_{\nu^{(n)}}(1^n)},
\end{equation}

\begin{equation} \label{p}
\frac{s_{\nu^{(n)}}(u_1,\dots,u_{n-1},1)}{s_{\nu^{(n)}}(1^n)}=\sum_{\mu^{(n-1)}\in\GT_{n-1}}\pr(\nu^{(n)} \rightarrow \mu^{(n-1)}) \frac{s_{\mu^{(n-1)}}(u_1,\dots,u_{n-1})}{s_{\mu^{(n-1)}}(1^n)},
\end{equation}

where $1^n$ is a notation for a string $(1,1,\dots, 1)$ of length $n.$

  \begin{lem} \label{odin} The following equalities hold

\begin{equation} \label{st}
 \st(\mu^{(n)} \rightarrow
 \nu^{(n)})=\begin{cases}\frac{s_{\nu^{(n)}}(1^n)}{2^{n}s_{\mu^{(n)}}(1^n)}, & \mu^{(n)}\subset
   \nu^{(n)} ; \\
 0, & {otherwise},\end{cases}
\end{equation}

\begin{equation} \label{pr}
 \pr(\nu^{(n)} \rightarrow
 \mu^{(n-1)})=\begin{cases}\frac{s_{\mu^{(n)}}(1^n)}{s_{\nu^{(n)}}(1^n)}, & \mu^{(n-1)}\prec
   \nu^{(n)}; \\
 0, & {otherwise.}\end{cases}
\end{equation}

As a consequence,
  \begin{equation}
\sum\limits_{\nu^{(n)}\in \GT_n}\st(\mu^{(n)} \rightarrow \nu^{(n)})=1 \text{ and }
\sum\limits_{\mu^{(n-1)}\in\GT_{n-1}}\pr(\nu^{(n)} \rightarrow \mu^{(n-1)})=1.
  \end{equation}

  \end{lem}
 \begin{proof}

Let us start with the first identity.
Let $e_l$ be the $l$-th elementary symmetric function.
Recall the Pieri's rule
$$e_l s_{\mu}=\sum\limits_{\mu\subset\lambda}s_{\lambda},$$
where the sum is taken over $\lambda$ obtained from $\mu$ by adding $l$ boxes,
no two in the same row.

 Note that
$$\displaystyle\prod\limits_{j=1}^{n}\frac{(1+u_j)}{2}=\frac{1}{2^{n}}\sum\limits_{j=1}^{n}e_{j}(u_1,\dots,u_{n}).$$

Therefore,
\begin{equation*}
 \st(\mu^{(n)} \rightarrow
 \nu^{(n)})=\begin{cases}\frac{s_{\nu^{(n)}}(1^n)}{2^{n}s_{\mu^{(n)}}(1^n)}, & \mu^{(n)}\subset
   \nu^{(n)} ; \\
 0, & {otherwise},\end{cases}
\end{equation*}
where $\nu^{(n)}$ is obtained from $ \mu^{(n)}$ by adding $l$ boxes such that $l\leq n.$

We can also compute

\begin{equation} \label{e}
s_{\mu^{(n)}}(1^n)e_l(1^n)=s_{\mu^{(n)}}(1^n) \binom{n}{l}=\sum\limits_{\mu^{(n)} \subset\lambda}s_{\lambda}(1^n),
\end{equation}
 where the sum is taken over $\lambda$ obtained from $\mu^{(n)}$ by adding $l$ boxes,
no two in the same row.

From (\ref{e}) we conclude
$$\sum\limits_{\nu^{(n)}\in \GT_n}\st(\mu^{(n)} \rightarrow \nu^{(n)})=\frac{1}{2^{n}}\sum\limits_{l=1}^{n} \binom{n} {l} =1.$$

The equation (\ref{p}) is a well-known branching rule for the Schur polynomials. Thus, we get

\begin{equation*}
 \pr(\nu^{(n)} \rightarrow
 \mu^{(n-1)})=\begin{cases}\frac{s_{\mu^{(n)}}(1^n)}{s_{\nu^{(n)}}(1^n)}, & \mu^{(n-1)}\prec
   \nu^{(n)}; \\
 0, & {otherwise.}\end{cases}
\end{equation*}

Moreover, we
have $$\textup{dim} (\pi^{\mu^{(n)}})= s_{\mu^{(n)}}(1^n)=\sum\limits_{\nu^{(n)}\prec
  \mu^{(n)}}\textup{dim}(\pi^{\nu^{(n)}})=\sum\limits_{\nu^{(n)}\prec
  \mu^{(n)}}s_{\nu^{(n)}}(1^n).$$
 \end{proof}

Let us fix a signature $\mu$ of length $N.$ Lemma \ref{odin} allows us
to define a {\it probability measure} on the sequences of signatures of the
form
$$\mathcal S^N=\{(\mu^{(N)}, \nu^{(N)},\dots,\mu^{(1)}, \nu^{(1)})\}$$ by the formula
\begin{multline} \label{eq:measure1}
\mathcal P^N_{\mu}\left ((\mu^{(N)}, \nu^{(N)},\dots,\mu^{(1)},
  \nu^{(1)})\right )=\\
=1_{\mu^{(N)}=\mu}\st(\mu^{(N)} \rightarrow \nu^{(N)}) \prod\limits_{j=1}^{N-1}
 \pr(\nu^{(N-j+1)} \rightarrow \mu^{(N-j)})\st(\mu^{(N-j)} \rightarrow \nu^{(N-j)}),
\end{multline}
where $\mu^{(i)},\nu^{(i)} \in \GT_i. $

This measure is similar to a Schur process; the only difference is that
we fix a boundary condition $\mu$ (the factor $1_{\mu^{(N)}}=\mu$ in
  the formula \eqref{eq:measure1})
which does not fit the framework of Schur processes (boundary
conditions have to be empty there). It is this difference that
significantly distinguishes asymptotic analysis of our model (as well
as models in \cite{DM}, \cite{P1}, \cite{P2}) with the case of Schur processes.  

\begin{prop} \label{uniform} Let $\mathcal R(\a, N,m)$ be a rectangular Aztec
  diamond. Let us denote the support of measure $\mathcal P^N_{\omega}$ by $\mathbb S_{\omega},$ where $\omega$ is the
  signature corresponding to the boundary row of $\mathcal R(\a, N,m).$
The set of domino tilings $\mathfrak D(\a, N,m)$ of $\mathcal R(\a, N,m)$ is in bijection with $\mathbb S_{\omega}.$ Moreover, the measure $\mathcal P^N_\omega$ is uniform on $\mathbb S_\omega$ and, therefore, on $\mathfrak D(\a, N,m).$
\end{prop}

\begin{proof}
From (\ref{pr}) and (\ref{st}) for any $(\mu^{(N)}, \nu^{(N)},\dots,\mu^{(1)}, \nu^{(1)})$ we have
\begin{align}
&\mathcal P^N_{\omega}\left((\mu^{(N)}, \nu^{(N)},\dots,\mu^{(1)},
  \nu^{(1)})\right )=\notag \\
&=\frac{1}{2^N}\frac{s_{\nu^{(1)}}(1^{N})}{s_{\mu^{(1)}}(1^{N})}\frac{s_{\mu^{(2)}}(1^{N-1})}{s_{\nu^{(1)}}(1^{N})}
\frac{1}{2^{N-1}}\frac{s_{\nu^{(2)}}(1^{N-1})}{s_{\mu^{(2)}}(1^{N-1})}
\cdots \frac{1}{2} \frac{s^{\nu^{(N)}}(1)}{s_{\mu^{(N)}}(1)}=\notag \\
&=\frac{1}{2^{N(N+1)/2}s_{\omega}(1^{N})},   \notag
\end{align}
when $(\mu^{(N)}, \nu^{(N)},\dots,\mu^{(1)}, \nu^{(1)})\in \mathcal S(Y_{\omega}, N, m)$
and zero otherwise. Recall that by Theorem \ref{bij} the set
$\mathcal S(Y_{\omega}, N, m)$ is in bijection with the set of domino tilings $\mathfrak D(\a, N,m).$

Therefore, $\mathcal P^N_{\omega}$ defines the uniform measure on $\mathfrak D(\a, N,m).$

\end{proof}

\begin{cor}\label{partion}
The number of domino tilings of $\mathcal R(\a,N,m)$ is equal to
\begin{equation}
|\mathfrak D(\a,N,m)|=2^{N(N+1)/2}s_{\omega}(1^{N}).
\end{equation}
\end{cor}

This formula was obtained in \cite{C2} based on the results from
\cite{C1} and \cite{M}, see also \cite{HG}. There are many ways to
prove the corresponding formula in the case of the Aztec diamond, see \cite{EKLP}. 

\subsection{Non-intersecting line ensembles}
\label{sec:paths}
Let us describe in more detail another combinatorial interpretation of
our tiling model, which was discussed in the introduction.

One can imagine superimposing a rectangular Aztec diamond upon a checkerboard coloring and obtaining four types of dominos, as the black unit square might be either to the right/left (resp. top/bottom) unit square of a horizontal
(resp. vertical) domino. In this way to each tiling one can bijectively associate a set of non-intersecting
lines, as shown in Figure \ref{fig:lines}. This construction first appeared in \cite{J}, we use a slightly different but
equivalent representation from \cite{BF}.

 Next, one can think of a rectangular Aztec diamond as being embedded
into tilings of $\mathbb R^2$, where outside the domain we add only non-overlapping horizontal
dominos and fill the whole space with them. Then, by means of a simple transformation one obtains a bijection between a tiling of a rectangular
Aztec diamond $\mathcal R (N, \Omega, m)$ and a configuration of non-intersecting lines on
a Lindstr{\"o}m-Gessel-Viennot (LGV) directed graph (see \cite{FS} and
\cite{St}), built out of
$N+m$ basic blocks, see Figure
\ref{fig:LGV}. Note that $\Omega$ determines the boundary conditions
along the bottoms of the basic blocks as shown in Figure \ref{fig:LGV}. By construction from a yellow vertex a line can go up-right or to
the right, while from a red vertex the line goes to the right or
down-right.

\begin{figure}[h]
\centering
\begin{subfigure}{.5\textwidth}
  \centering
 \includegraphics[width=0.9 \linewidth]{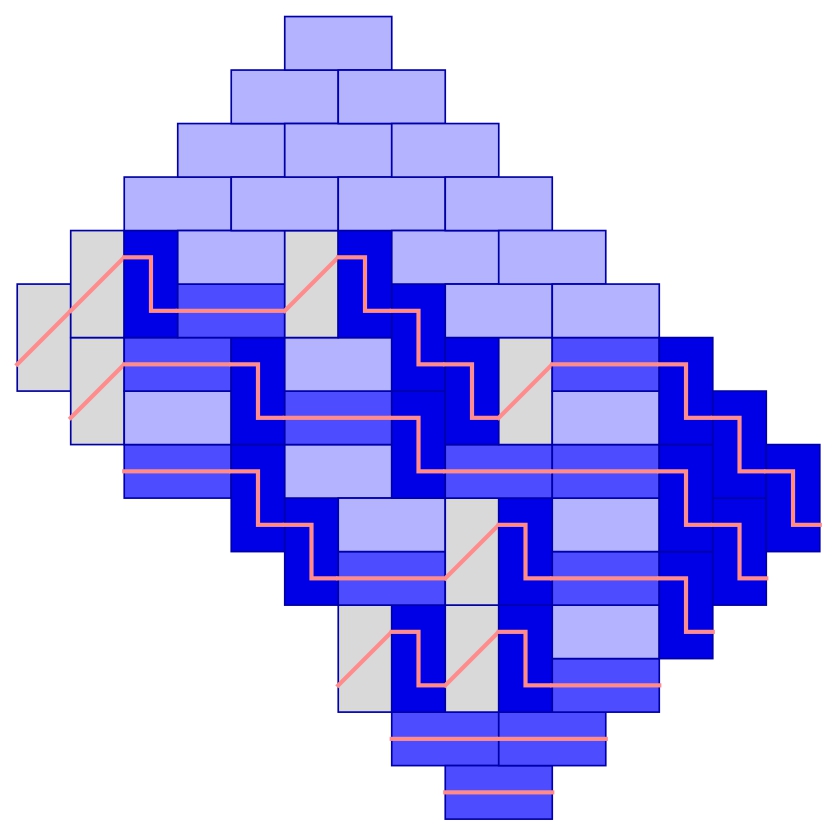}
 \label{fig:sub1}
\end{subfigure}%
\begin{subfigure}{.5\textwidth}
  \centering
\includegraphics[width=0.9 \linewidth]{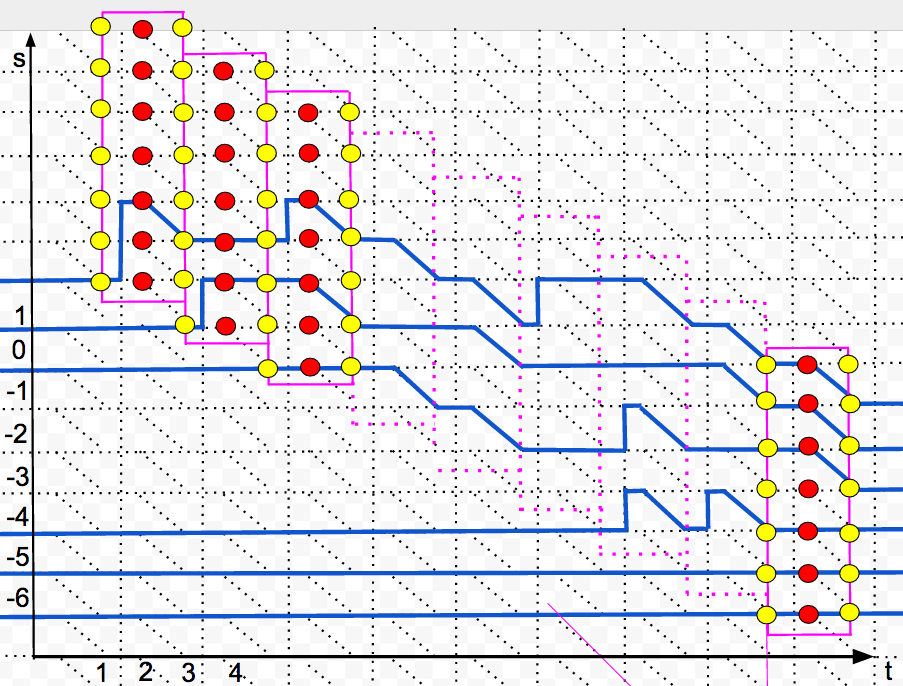}
  \label{fig:sub2}
\end{subfigure}
\caption{Rectangular Aztec diamond $\mathcal R(6, (1,2,3,7,8,9),3)$
  and the corresponding ensemble of non-intersecting paths. The height
of a block is $N+1$ units.}
\label{fig:LGV}
\end{figure}

Introduce the coordinate system as in
the right picture in Figure \ref{fig:LGV}. Let us associate to every
vertical section $t=i,$ where $i=1, \dots 2N+1,$ a signature. Consider the set of nodes
on a section. Suppose the nodes that belong to
the lines from the ensemble have coordinates $s_1>s_2>\dots>s_i.$ Then
the corresponding signature is $(s_1+1\geq s_2+2 \geq \dots\geq s_i+i).$ In this way to every
non-intersecting line ensemble one associates a sequence of signatures
$( \gamma_1, \rho_1,\dots, \gamma_N, \rho_N, \gamma_{N+1}).$ Note that $\gamma_{N+1}=(0\geq 0\geq
\dots \geq 0).$

Consider the following probability model on the set of the non-intersecting line
ensembles. Let us assign weight $1$ to each horizontal edge, weight
$\alpha=1$ to each vertical edge and $\beta=\frac{1}{2}$ to each
down-right edge. Consider the $i$-th basic
block, let $(\gamma_i, \rho_{i}, \gamma_{i+1})$ be a triple of signatures
corresponding to its left side, middle section and its right side.
Let us assign to this block the following weight
$$w(\text{block})=s_{\rho_{i}/\gamma_{i+1}}(\alpha)\cdot
s_{\rho_{i}/\gamma_i}(\hat {\b}),$$
where  $s_{\rho/ \gamma}$ is a skew Schur function.
Here $\alpha$ stands for the specialization of the symmetric functions into a single nonzero
variable equal to $\alpha$ (with complete homogeneous symmetric functions
specializing into $h_n(\alpha) = \alpha^n, n\geq 0)$, and $\hat {\b}$ stands for the specialization into
a single ``dual'' variable equal to $\b$ (with complete homogeneous symmetric
functions specializing into $h_n(\hat {\b}) = 1$ if $n = 0,$ $h_n(\hat {\b}) = \b$ if $n = 1$, and
$h_n(\hat {\b}) = 0$ if $n \geq 2)$. The quantities $w(\text{block})$ are essentially indicators
times a power of a parameter.

Define the probability $\mathbb P_L$ of an
non-intersecting line ensemble satisfying our combinatorial and
boundary conditions to be the product of the corresponding weights of
the blocks.

\begin{prop}
Measure $\mathbb P_L$ is a uniform probability measure on the set  $L(N, \Omega,
m)$ of all
non-intersecting line ensembles satisfying our combinatorial and boundary
conditions.
\end{prop}
 We do not go into details of the proof of this proposition since
we will not need it further, see the recent exposition in \cite{BCV}.

We conclude that the uniform measure on the set of tilings $\mathfrak D (N,
\Omega, m)$ corresponds under the above bijection to a measure
$\mathbb P_L$ on the ensembles of non-intersecting lines.

There are many other interesting tiling models that have an interpretation as
ensembles of non-intersecting lines. For example, random
tilings of a tower, discussed in \cite{BF}. More generally, the random dimer
model on a rail yard graph \cite{BBCCR} fits into the framework of the
Schur generating functions.  We believe that our
approach can be used to analyze the global behavior
for these models with arbitrary boundary conditions.

%% file: probheight.tex
\section{Law of Large numbers}
\label{sec:3}

The main technique we use in the paper is the moment method introduced by Bufetov
and Gorin in \cite{BG} and \cite{BG2}.
In this section we first present some background to give an overview
of the method and then show that the
model of random domino tilings of rectangular Aztec diamonds fits into
this framework. Subsequently, we prove the Law of Large numbers for the
random height function.

\subsection{LLN for the moments}

One way to encode a signature $\lambda\in \GT_N$ is through the
\emph{counting measure} $m[\lambda]$ on $\mathbb R$
corresponding to it via
\begin{equation}
\label{eq_def_uniform_intro}
m[\lambda]=\frac{1}{N}\sum_{i=1}^N
 \delta\left(\frac{\lambda_i+N-i}N\right),
\end{equation}
where $\delta$ is a delta measure.
Note that we incorporate the scaling into this formula.

Let $\rho$ be a probability measure on the set of all signatures $\GT_N.$ The pushforward of $\rho$ with respect to the map
$\lambda\rightarrow m[\lambda]$ defines a \textit{random} probability
measure on $\mathbb R$ that we denote $m[\rho].$

\begin{defi} Let $\rho$ be a probability measure on $\GT_N.$ The
  Schur generating function $\mathcal S^{U(N)}_{\rho}(u_1,\dots, u_N)$ is a symmetric Laurent power series in $(u_1,\dots,u_N)$ given by
\begin{equation}
\mathcal S^{U(N)}_{\rho}(u_1,\dots, u_N)=\sum \limits_{\lambda \in \GT(N)} \rho(\lambda)\frac{s_{\lambda}(u_1,\dots,u_N)}{s_{\lambda}(1,\dots,1)}.
\end{equation}
\end{defi}

We will need the following two results from \cite{BG}.

\begin{thm}
\label{theorem_moment_convergence} $($\cite{BG}, Theorem 5.1 $)$
Let $\rho(N)$ be a sequence of measures such that for each $N=1,2,\dots$, $\rho(N)$ is a
probability measure on $\GT_N$. Suppose that $\rho(N)$ is such that for every $j$
$$
 \lim_{N\to\infty} \frac{1}{N} \ln( \mathcal S^{U(N)}_{\rho(N)}(u_1,\dots,u_j, 1^{N-j}) ) = Q(u_1)+\dots+Q(u_j),
$$
where $Q$ is an analytic function in a neighborhood of $1$ and the convergence is uniform in an
open $($complex$)$ neighborhood of $(1,\dots,1)$. Then random measures $m[\rho(N)]$ converge as
$N\to\infty$ in probability, in the sense of moments to a \emph{deterministic} measure $\mes$ on
$\mathbb R$, whose moments are given  by
\begin{equation}
\label{eq_limit_moments_A}
 \int_{\mathbb R} x^j \mes(dx)= \sum_{\ell=0}^j \frac{j!}{\ell! (\ell+1)! (j-\ell)!} \frac{\partial^\ell}{\partial u^\ell}\left(u^j
 Q'(u)^{j-l}\right)\Biggr|_{u=1}.
\end{equation}
\end{thm}

Let us introduce some notation.
Let $\eta$ be a compactly supported measure on $\mathbb
  R.$  Let  $M_k(\mes)=\int_{\mathbb R} x^k \mes(dx)$ be its $k$-th
  moment. Set
\begin{equation}\label{eq:generating}
S_\eta(z)=z+M_1(\mes)z^2+z^3 M_2(\mes) +\dots
\end{equation}
 to be the generating function of the moments of $\eta.$
Define $S_\eta^{(-1)}(z)$
 to be the inverse series to $S_\eta(z)$, that is $ S_\eta^{(-1)}\bigl( S_\eta(z)\bigr)=S_\eta\bigl(S_\eta^{(-1)}(z)\bigr)=z.$

Let
\begin{equation}
\label{eq_Voiculescu_R}
 R_\eta(z)= \frac{1}{S_\eta^{(-1)}(z)} -\frac{1}{z}
\end{equation}
be the \textit{Voiculescu $R$--transform} of measure $\eta.$

Define the function $\H_\eta(u)$ as
\begin{equation}
\label{eq_H_function} \H_\eta(u)=\int_0^{\ln(u)}R_\eta(t)dt+\ln\left(\frac{\ln(u)}{u-1}\right),
\end{equation}
which should be understood as a power series in $(u-1)$.

Note that we have the following expression for its derivative:
\begin{equation}\label{eq_H_derivative}
\H'_{\eta} (u) = \frac{1}{u S_{\eta}^{(-1)} (\log u)} - \frac{1}{u -1}.
\end{equation}

 \begin{lem} If $\eta$ is a measure with compact support, then $\H_\eta(u)$ as a power series in $(u-1)$ is uniformly convergent in an open
neighborhood of $u=1$.
\end{lem}
\begin{proof} Immediately follows from the definitions. \end{proof}

We will need the following mild technical assumption.
\begin{defi}
\label{definition_regularity}
 A sequence of signatures $\lambda(N)\in \GT_N$ is called \emph{regular}, if there exists a
 piecewise--continuous function $f(t)$ and a constant $C$ such that
\begin{equation}
\label{eq_reg1}
 \lim_{N\to\infty} \frac{1}{N}\sum_{j=1}^{N} \left|\frac{\lambda_j(N)}{N}-f(j/N)\right|=0
\end{equation}
and
\begin{equation}
\label{eq_reg2} \left|\frac{\lambda_j(N)}{N}-f(j/N)\right|<C,\quad \quad j=1,\dots,N,\quad
N=1,2,\dots.
\end{equation}
\end{defi}

\begin{rem}
Informally, the condition \eqref{eq_reg1} means that scaled by $N$
coordinates of $\lambda(N)$ approach a limit profile $f$. The restriction that $f(t)$ is
piecewise--continuous is reasonable, since $f(t)$ is a limit of monotone functions and, thus, is
monotone $($therefore, we only exclude the case of countably many
points of discontinuity for $f$$)$.
We use condition \eqref{eq_reg2} since it guarantees that all the measures which we assign to
signatures and their limits have $($uniformly$)$ compact supports --- thus, these measures are uniquely
defined by their moments.
\end{rem}

\begin{thm}
\label{Theorem_character_asymptotics}$($\cite{BG}, \cite{GP},\cite{GM}$)$
 Suppose that $\lambda(N)\in \GT_N$, $N=1,2,\dots$ is a regular sequence of signatures, such that
 $
  \lim_{N\to\infty} m[\lambda(N)]=\mes \text{ (in weak topology)}.$
 Then for any
 $j=1,2,\dots, N$ we have
 \begin{equation}
 \label{eq_limit_of_logarithm}
  \lim_{N\to\infty} \frac{1}{N}
  \ln\left(\frac{s_{\lambda(N)}(u_1,\dots,u_j,1^{N-j})}{s_{\lambda(N)}(1^N)}\right)=
  \H_\mes(u_1)+\dots+ \H_\mes(u_j),
 \end{equation}
 where the convergence is uniform over an open $($complex$)$ neighborhood
 of $(1,\dots,1).$
\end{thm}

The uniform measure on the set of domino tilings  of a
rectangular Aztec diamond $\mathcal R(N, \a, m)$ induces a
measure on the set of all possible configurations of $N-\lfloor\frac{k-1}{2}\rfloor$ $V$-squares
in the $k$-th row, for $k=1,\dots, 2N.$ Due to Theorem \ref{uniform} we
can think of it from the other prospective, more precisely, as a
measure on the signatures $\lambda \in \GT_{N-\lfloor\frac{k-1}{2}\rfloor}.$  Let us denote the measure we
get on the set of signatures by
$\rho^k(N)$ and $\lfloor\frac{k-1}{2}\rfloor$ by $t.$

\begin{lem} \label{Schur_generating}
\begin{equation}
\begin{dcases}
  \mathcal S^{U(N)}_{\rho^{k}(N)}(u_1,\dots, u_{N-t})=\frac{s_\omega(u_1,\dots,u_{N-t}, 1^{t})}{s_\omega(1^{N})} \prod\limits_{j=1}^{N-t}\left(\frac{1+u_{j}}{2}\right)^{t},       & k=2t+1, \text{ }
   t=0, \dots, N-1;\\
   \mathcal S^{U(N)}_{\rho^{k}(N)}(u_1,\dots,
   u_{N-t})=\frac{s_\omega (u_1,\dots,u_{N-t}, 1^{t})}{s_\omega(1^{N})}\prod\limits_{j=1}^{N-t}\left(\frac{1+u_{j}}{2}\right)^{t+1},
   & k=2t+2, \text{ } t=0, \dots, N-1.
  \end{dcases}
\end{equation}
\end{lem}

\begin{proof}
This statement is an immediate corollary of Proposition \ref{uniform}.
\end{proof}

We consider $N\rightarrow \infty$ asymptotics such that all the
 dimensions of $\mathcal R(N, \a(N), m(N))$ linearly grow with
 $N$. Assume that the sequence of signatures $\omega(N)$ corresponding to the
 first row is regular and $ \lim\limits_{N\to\infty} m[\omega(N)]=\mes_{\omega}.$
Then from Theorem \ref{Theorem_character_asymptotics} it follows that
for any $j=1,2,\dots, N$
\begin{equation}
 \label{eq_limit_of_log}
  \lim_{N\to\infty} \frac{1}{N}
  \ln\left(\frac{s_{\lambda(N)}(u_1,\dots,u_j,1^{N-j})}{s_{\lambda(N)}(1^N)}\right)=
  \H_{\mes_{\omega}}(u_1)+\dots+ \H_{\mes_{\omega}}(u_j).
 \end{equation}
Let us
 look at what is happening at the $k$-th row, we assume that
 $k=[2 \kappa N ]$ and $\kappa<1.$ It has $N-t=\lfloor \frac{k-1}{2} \rfloor$
 $V$-squares. We have
\begin{multline}
\lim_{N\to\infty} \frac{1}{(1-\kappa)N} \log \left( \mathcal
S^{U(N)}_{\rho^k(N)}(u_1,\dots,u_{j}, 1^{N-t-j} )\right)=\\ =\lim_{N\to\infty}
\frac{1}{(1-\kappa) N}\log \left(\frac{s_{\w}(u_1,\dots,u_{j}, 1^{N-t-j} )}{s_{\omega}(1^N)}
\prod\limits_{i=1}^{ j}\left(\frac{1+u_{i}}{2}\right)^{t / t+1} \right)=\\
=\sum\limits_{i=1}^{j}\left(\frac{\kappa}{1-\kappa}\log \left (\frac{1+u_i} {2}\right)+\frac{1}{1-\kappa}\H_{\mes_\omega}(u_i)\right).
\end{multline}
Thus, from Theorem \ref{theorem_moment_convergence} we get the following proposition.

\begin{prop} \label{ourmoment_convergence}
 Assume that the sequence of signatures $\omega(N)$ corresponding to the
 first row is regular and $ \lim\limits_{N\to\infty}
 m[\omega(N)]=\mes_\omega$ $($ weak convergence$)$.
Random measures $m[\rho^k(N)]$ converge as
$N\to\infty$ in probability, in the sense of moments to a \emph{deterministic} measure $\mes^\kappa$ on
$\mathbb R$, whose moments are given  by
\begin{equation}
\label{eq_limit_moments_A}
 \int_{\mathbb R} x^j \mes^\kappa(dx)= \frac{1}{2(j+1)\pi\ii} \oint_{1} \frac{dz}{z}
 \left(zQ'(z)+\frac{z}{z-1}\right)^{j+1},
\end{equation}
where $Q'(u)=\frac{1} { 1-\kappa }\left (\frac{\kappa}
{1+u}+\H'_{\mes_\omega}(u) \right)$ and the integration goes over a small positively oriented contour around $1$.
\end{prop}
\begin{proof}
From Theorem \ref{theorem_moment_convergence} using the integral
representation of the derivative we get
 \begin{multline}
\label{eq_moment_integral_formula}
 M_j(\mes^k)=\sum_{\ell=0}^j \frac{j!}{\ell! (\ell+1)! (j-\ell)!} \frac{\partial^\ell}{\partial u^\ell}\left(u^j
 Q'(u)^{j-l}\right)\Biggr|_{u=1}\\=\frac{1}{2\pi\ii} \oint_{1}  \sum_{l=0}^j  \frac{1}{l+1} {j\choose l} \frac{z^j
 Q'(z)^{j-l}}{(z-1)^{l+1}} dz =
\\\quad \text{(we can add a holomorphic function to the integrand)}
 \\= \frac{1}{2\pi\ii} \oint_{1} \frac{z^j Q'(z)^{j+1}}{j+1} \sum_{l=-1}^j  {{j+1} \choose {l+1}}
 \frac{1}{
 Q'(z)^{l+1}(z-1)^{l+1}} dz\\= \frac{1}{2\pi\ii} \oint_{1}  \frac{z^j Q'(z)^{j+1}}{j+1}
 \left(1+\frac{1}{Q'(z)(z-1)}\right)^{j+1} dz\\= \frac{1}{2(j+1)\pi\ii} \oint_{1} \frac{dz}{z}
 \left(zQ'(z)+\frac{z}{z-1}\right)^{j+1}.
\end{multline}
\end{proof}

\begin{rem}
\label{rem:lln}

Proposition \ref{ourmoment_convergence} implies the existence of a
limit shape for the model we are considering. As we noted in the
introduction, for a uniform measure on tilings this
result is not new, it follows from \cite{CKP}, \cite{KO}. However, Proposition
\ref{ourmoment_convergence} supplies explicit formulas for the answer
which allows the further study of properties of limit shapes. Also, a
similar analysis is possible in the case of some non-uniform measures, see
\nameref{sec:A}.

We can put the computation of this limit shape into the
context of the \textit{quantized free convolution}. This notion
$($introduced in \cite{BG}$)$ is closely but non-trivially related to the
notion of the free convolution, a well-known operation introduced by
D. Voiculescu $($see e.g. \cite{VDN}$)$. The computation of the limit
shape in our case results in the following algorithm: we start with an
arbitrary measure with compact support $\mes_w$, then consider a
semigroup $\mes_w \otimes t \mathcal{B}$ $($the quantized free
convolution of two measures$)$, where $t>0$ and $\mathcal B$ is
an \textit{extreme} beta character, and  finally compute quantized free \textit{projections} of the measures from this semigroup.
We refer to \cite{BG} for definitions of these operations and notions.

The relation between tiling models and free probability is very
interesting and yet to be well-understood. See Remark \ref{rem:clt} for further comments.

\end{rem}

\subsection{Height function}
\label{sec:height}
\begin{defi} \label{defi_height} \cite{T} 
Let $D\in \mathfrak D(N,\a,m)$ be a domino tiling of a rectangular Aztec diamond $\mathcal R(N,\a, m).$
Let us impose checkerboard coloring on the plane $(i,j)$ such that the boundary row of $\mathcal R(N,\a, m)$ consists of the dark squares, see Figure \ref{fig:height}. A height function $h_D$ is an integer-valued function on the vertices of the lattice squares of $\mathcal R(N,\a, m)$ which satisfies the following properties:
\begin{itemize}
\item if an edge $(u,v)$ does not belong to any domino in $D$ then $h(v)=h(u)
+1$ if $(u,v)$ has a dark square on the left, and $h(v)=h(u)
-1$ otherwise.
\item if an edge $(u,v)$ belongs to a domino in $D$ then $h(v)=h(u)
+3$ if $(u,v)$ has the dark square on the left, and $h(v)=h(u)-3$ otherwise.
\item $h_D(u_0)=0$, where $u_0$ is the vertex in the upper right corner of  $\mathcal R(N,\a, m),$ see Figure \ref{fig:height}.
 \end{itemize}
 \end{defi}

 \begin{figure}[h]
\includegraphics[width=0.45\linewidth]{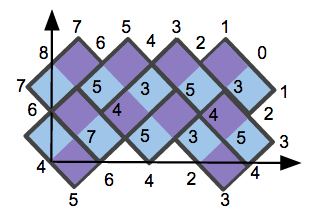}
\caption{An example of a domino tiling of  $\mathcal R(2, \a=(1,4), 2)$ and its height function.  }
  \label{fig:height}
\end{figure}

Fix a domino tiling $D\in \mathfrak D(N, \a, m)$. Note that the height function is uniquely determined on the boundary
of the domain. Let us calculate it inside.
Consider a vertex $(i,j)$ inside the domain. It is the left-most
vertex of a unit square that belongs to a row with number
$(2j+1).$ Such a row contains $(N-\lfloor j \rfloor)$
$V$-squares. Let
$Y_{\mu^{D, j}}$ be the Young
diagram corresponding to this row under the Construction \ref{con},
where $\mu^{D, j}=(\mu^{D, j}_1\geq\mu^{D, j}_2 \geq \dots\geq \mu^{D,
  j}_{N-\lfloor j \rfloor}).$ Define the function

\begin{equation}\label{eq:delta}
\Delta^N_D(i,j)\colon = \left| \left\{ 1 \le s \le N- \lfloor j \rfloor: \mu^{D, j}_s +(N -\lfloor j \rfloor)- s \ge i \right\}\right|.
\end{equation}

Denote by $t_k$ the position of the $k$-th $V$-square in this row counting from the left.
By construction we have $$\mu^{D, j}_s+(N-\lfloor j \rfloor)-s=t_{N-\lfloor j \rfloor-s+1}-1.$$
 Thus, $\Delta^N_D(i,j)$ simply counts the number of $V$-squares to the right from the vertex $(i,j)$ in the corresponding row.
Note that we can rewrite it in terms of the measure $m[\mu^{D, j}]$
 \begin{equation}
\Delta^N_D(i,j)=(N-\lfloor j \rfloor)\int 1_{[ \frac{
    i}{N-\lfloor j \rfloor},\infty]} dm[\mu^{D, j}].
\end{equation}

\begin{lem}  \label{height_formula}
Let $D\in \mathfrak D(N,\a,m)$ be a domino tiling of the rectangular
Aztec diamond $\mathcal R(N,\a, m).$ Let $(i, j)$ be a vertex inside
the domain. Then
 \begin{equation}\label{heighteq}
   h_D((i,j))= 2\left( 2N+m-j-\lfloor i\rfloor-2\Delta^N_D(i,j) \right),\text{ } j=0,\dots, N.
 \end{equation}

\end{lem}

\begin{proof}
Consider a path from $(-\frac{1}{2},j)$ to $(N+m+\frac{1}{2},j)$ for
the even rows and a path from $(0,j)$ to $(N+m,j)$ for the odd rows
such that it goes from vertex
$(i,j)$ to $(i+1, j)$ along the boundary of a domino $d\in D,$ see
Figure \ref{fig:height2}.

 \begin{figure}[h]
\includegraphics[width=0.3\linewidth]{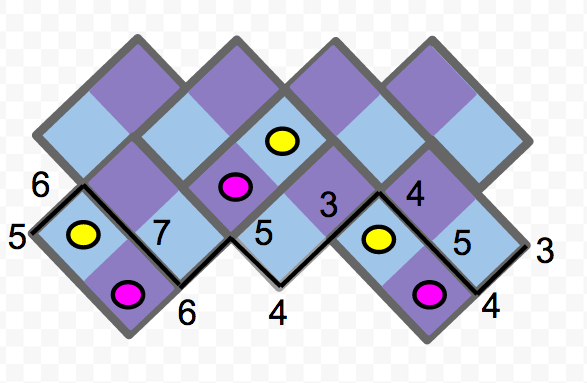}
\caption{An example of a domino tiling of  $\mathcal R(2, \a=(1,4), 2)$ and its height function.  }
  \label{fig:height2}
\end{figure}

By the definition of the
height function with every step along this path it changes by one. The
sign of this change depends on whether the path goes along the
boundary of a $V$-square or a $\Lambda$-square. Let us follow this
path from right to left. Then the height function increases by two if the path
goes along the boundary of a $\Lambda$-square and decreases by two
otherwise. Note that the number of squares to the right
from the vertex $(i,j)$ is equal to $(N+m-\lfloor
 i\rfloor).$ So the total increment
of the height function is equal to

 $$2\left(\#\{\Lambda-\text{squares}\}-\#\{V-\text{squares}\}\right)=2\left((N+m-\lfloor
 i\rfloor)-2\#\{V-\text{squares}\}\right).$$

Thus, what is left is to compute the value of the height
function on the right boundary of the domain. It is easily done by
going alone the boundary starting from the point $u_0.$ We see that it is equal
to $2(N-j).$
 \end{proof}

\begin{thm}  \label {LLN} $($Law of Large numbers for the height function.$)$
Consider $N\rightarrow \infty$ asymptotics such that all the
 dimensions of a rectangular Aztec diamond $\mathcal R(N, \a(N), m(N))$ linearly grow with
 $N$.
 Assume that the sequence of signatures $\omega(N)$ corresponding to the
 first row is regular, $ \lim\limits_{N\to\infty}
 m[\omega(N)]=\mes_\omega$ $($weak convergence$)$ and  $m(N)/N \rightarrow \nu\in \mathbb R_{>0}$ as $N$ goes to infinity. Let us fix $\kappa \in (0, 1)$ and let
$\mes^\kappa$ be the limit of measures $m[\rho^k(N)],$ which is given by Proposition \ref{ourmoment_convergence}.

Define
$$ {\bold h}(\chi,\kappa):=2\left(2+\nu-\kappa-\chi-2(1-\kappa)
  \int\limits_{\frac{\chi}{1-\kappa}}^{\infty}{\bold d}\mes^\kappa(x) \right).$$
Then the random height function $h_{\mathfrak D}$ converges uniformly in
probability to a deterministic function $ {\bold h}(\chi,\kappa)$:
 $$\frac{h_{\mathfrak D}([ \chi N ], [ \kappa N ])}{N} \rightarrow {\bold h}(\chi,\kappa), \text{ as } N\rightarrow \infty, $$
 where $(\chi,\kappa)$ are the new continuous parameters of the
 domain.
\end{thm}

 \begin{proof}
From Lemma \ref{height_formula} we see that for a fixed $j=[ \kappa
N]$ the height function is the
scaled distribution function of the measure $m [\rho^k(N)]$ up to a constant, where $k$ is the
number of the corresponding row. Thus, from Proposition
\ref{ourmoment_convergence} the Law of Large numbers follows, see
e.g. Proposition 2.2 \cite{BBO}.

 \end{proof}

%% file: gen_measure_new.tex
 \section{Properties of the limit measure }
\label{sec:4}
Let us fix some notation. We consider $N\rightarrow \infty$ asymptotics such that all the
 dimensions of a rectangular Aztec diamond $\mathcal R(N, \a(N), m(N))$ linearly grow with
 $N$. The limit domain scaled by $N$ is denoted $\mathcal R$ and
 the new continuous coordinates inside the domain are $(\chi,
 \kappa).$ We also assume that the sequence of signatures $\omega(N)$ corresponding to the
 first row is regular (see Definition \ref{definition_regularity}) and the sequence of measures $\{
 m[\omega(N)]\}$ weakly converge to $\mes_\omega.$ In this section our
 goal is to compute the density of the limit measures $\mes^\kappa,$
 defined in Proposition \ref{ourmoment_convergence} .

\label{sec:4}
\subsection{Stieltjes transform of the limit measure}
Note that under our assumptions $\mes_\omega$ has compact
support. Moreover, it is absolutely continuous with respect to the Lebesgue measure and its density takes values in $[0,1].$

Recall that the Stieltjes transform of a compactly-supported measure $\eta$
is defined as
$$\textup{St}_\eta(t)=\int\limits_{\mathbb R} \frac{\eta[ds]}{t-s},$$
for  $t\in \mathbb C\setminus \textup{Support}(\eta).$

If the measure $\eta$ has moments $M_k$ of any order $k$
then the Stieltjes transform admits for each integer $n$ an
asymptotic expansion in the neighbourhood of infinity given by
$$\textup{St}_\eta(t)=\sum_{k=0}^{n}\frac{M_k}{t^{k+1}}+o\left(\frac{1}{t^{k+1}}\right).$$

In particular, in this case there is the following connection between the moment
generating function $S_\eta$ \eqref{eq:generating} and the Stieltjes transform $\textup{St}_\eta$:
$$S_{\eta}\left(\frac{1}{t}\right)=\textup{St}_\eta(t)$$
for $t$ in the neighborhood of infinity.

Let $x\in \mathbb C$ and $\kappa \in (0,1).$ Consider the following
system of equations in $z$ and $t,$ where $z\in \mathbb C\setminus
\mathbb R_-$ and $t\in \mathbb C\setminus
\textup{Support}(\mes_\omega)$

\begin{equation}\label{sys}
\begin{cases}
\F_{\kappa}(z{, }t)=x, \\
 \textup{St}_{\mes_\omega}(t) =\log(z),
\end{cases}
\end{equation}

 where we consider the principal branch of the logarithm and
\begin{equation} \label{eq:F}
\F_{\ka}(z,t):=\frac{z}{(1-\ka)}\Big (\frac{t}{z}- \frac{1}{(z-1)}+\frac{\kappa}{1 + z} \Big ) + \frac{z}{z - 1}.
\end{equation}

This system of equations will provide formulas for the limit shape. We
start
the analysis of (\ref{sys}) considering the case when $x$ is
sufficiently large.
The formal power series formalism will be useful for us.

Recall that whenever a formal power series of the form $ g(u)=\sum\limits_{i=1}^{\infty} g_i u^i \in \mathbb R[[u]]$
has  $g_1\neq 0,$ there exists a formal power series
$ h(u)=\sum\limits_{i=1}^{\infty} h_i u^i \in \mathbb R[[u]]$
that is a unique composition inverse of
$g(u),$ meaning that $g(h(u))=u.$

In the case of Laurent series of the form
$p(u)=\frac{1}{u}+\sum\limits_{i=0}^{\infty}p_i u^i \in \mathbb R((u))$ there also exists
a unique composite inverse Laurent series
$q(u)=\sum\limits^{\infty}_{i=1} q_i (\frac{1}{u})^i \in \mathbb R((u))$
such that $p(q(u))=u.$

Note that $\textup{St}_{\mes_\w}(t)$ is a uniformly convergent power series in
$\frac{1}{t}$ for $t$ in a neighborhood of infinity. Let $$t(z):=\textup{St}^{(-1)}_{\mes_\w}(\log z),$$ which we
view as a Laurent series in $(z-1).$ More precisely, we have
$$t(z)=\frac{1}{z-1}+\sum\limits_{i=0}^{\infty}u_i(z-1)^i, \quad
u_i\in \mathbb C.$$
Substituting this formula into $\F_\kappa(z, t)$ (\ref{sys}) we
obtain that $\F_\kappa(z, t(z))$ is of the following form as Laurent series in $(z-1):$
$$\F_\kappa(z, t(z))=\frac{1}{z-1}+\sum\limits_{i=0}^{\infty}f_i (z-1)^i, \quad
f_i\in \mathbb C.$$
This series is uniformly convergent for $z$ in a punctured neighborhood of $1.$

Define $z^{\kappa}(x)$ as the composite inverse to $\F_\kappa(z, t(z)).$
Note that $z^{\kappa}(x)$ is formal power series in $\frac{1}{x},$
which is uniformly convergent for $x$ in a neighborhood of infinity.

\begin{lem} \label{series}
The following formula is valid for $x$ in a neighborhood of infinity
$$\textup{St}_{\mes^\kappa}(x)=\log(z^{\kappa}(x)),$$
where $\textup{St}_{\mes^\kappa}(x)$ is the Stieltjes transform
of the limit measure $\mes^\kappa,$ defined in Proposition \ref{ourmoment_convergence}.
\end{lem}

\begin{proof}

We know from Proposition \ref{ourmoment_convergence} that the $j$-th moment of measure $\mes^\kappa$ can be computed as
$$ M_j (\mes^\kappa) = \frac{1}{2 (j+1) \pi \ii} \oint_{1} \frac{d z}{z} \left(
 \frac{z}{(1-\ka)}\left (\H'_{\mes_\omega}(z)
+\frac{\kappa}{1 + z}\right )+  \frac{z}{z-1} \right)^{j+1},
$$
where the contour of integration is a circle of radius $\epsilon \ll 1$ around $1.$

Using the formula (\ref{eq_H_derivative}) 
$$\H'_{\mes_\omega} (z) = \frac{1}{u S_{\mes_\omega}^{(-1)} (\log z)} - \frac{1}{z -1}.$$
for $\H'_{\mes_\omega} $ we get

$$\frac{z}{(1-\ka)}\left (\H'_{\mes_\omega}(z)
+\frac{\kappa}{1 + z}\right)+  \frac{z}{z-1} =\F_\kappa(z, t(z)),$$
where $t(z)=\textup{St}^{(-1)}_{\mes_\w}(\log z).$

Now we can compute the Stieltjes transform for $x$ in the neighborhood of infinity as
$$
 \sum_{j=0}^{\infty} \frac{M_j(\mes^\kappa)}{x^{j+1}} = -\frac{1}{2
   \pi \ii} \oint_{1} \frac{dz}{z} \log \left( 1 - \frac{\F_\ka(z, t(z))}{x} \right).
$$

Integrating by parts we get
$$
\sum_{j=0}^{\infty} \frac{M_j(\mes^\kappa)}{x^{j+1}}= \frac{1}{2 \pi
  \ii} \oint_{1} \log(z) \frac{\pa_z \left( 1 - \frac{\F_{\ka}(z,
      t(z))}{t} \right)}{1 - \frac{\F_{\ka}(z, t(z))}{x}}.
$$

The integrand has poles at roots of $\F_{\ka}(z, t(z))=x.$ For $x$ large
  enough the only root inside the region of integration is given by a
  unique composite inverse series of the Laurent series $\F_{\ka}(z, t(z))$ for $z$ in the
neighborhood of $1.$
 We need to compute the residue
to obtain the final result. Thus, we get
$$\textup{St}_{\mes^\kappa}(x)=\log(z^{ \kappa}(x)) $$
for $x$ in the neighborhood of infinity. Another way to justify the last step is to use the notions of residues and integration by parts for formal power series rather than integrals.
\end{proof}

 \subsection{The density of $\mes^\kappa$}

Let us define a function
$${\bold
  Z}^{ \kappa}(x)=\exp(\textup{St}_{\mes^\kappa}(x)), $$
where $x\in \mathbb C \setminus \textup{Support}(\mes^\kappa).$

Note that due to Lemma \ref{series}
\begin{equation} \label{eq:root}
\F_{\kappa}({\bold
  Z}^{\kappa}(x), t({\bold
  Z}^{\kappa}(x)))-x=0
\end{equation}
for $x$ in the neighborhood of
infinity. Since the left-hand side of this equation is an analytic
function on its domain of definition,
the equality holds for any $x\in \mathbb C \setminus \textup{Support}(\mes^\kappa).$
Therefore, ${\bold
  Z}^{ \kappa}(x)$ is one of the roots of (\ref{sys}). The system of equations (\ref{sys}) may have several roots. A certain effort is necessary in order to determine which of them should be chosen. We deal with this problem in the rest of this section.

Let us recall a well-known fact about Stieltjes transform of a
measure.
\begin{lem} \label{support}
If a measure $\eta$ has a continuous density $f(x)$ with respect to the
Lebesgue measure then the following formula is valid
$$f(x)=-\lim\limits_{\varepsilon\rightarrow
  0^+}\frac{1}{\pi}\textup{Im}(\textup{St}_{\eta}(x+i\varepsilon)).$$
\end{lem}

\begin{thm}  \label{density}
The density
of $\mes^\kappa$ is given by
\begin{equation} \label{eq:density_formula}
{\bold d} \mes^\kappa(x)=\frac{1}{\pi}\textup{Arg}( {\bold
  z}^\kappa_+(x)),
\end{equation}
where ${\bold z}^\kappa_+(x)$ is the unique complex root of the system
$($\ref{sys}$)$ which belongs to the upper half-plane. This
formula is valid for such $(x, \kappa)$ that the complex root
exists, the density is equal to zero or one otherwise.
\end{thm}

\begin{proof}  We will first prove this statement for a special case of measure $\mes_\omega.$ Let us fix a natural number $s$ and let $(a_1, a_2, \dots, a_s)$ and
 $(b_1, b_2, \dots, b_s)$ be two $s$-tuples of real numbers such that
 $$ a_1< b_1<a_2<b_2\dots <a_s<b_s \text{ and } \sum\limits_{i=1}^s(b_i-a_i)=1.$$
 Assume that $\mes_\omega$ is a uniform measure on the union of the intervals $[a_i, b_i].$

Then we can compute $$\textup{St}_{\mes_\omega}(t)=\log\frac{(t-a_1) \cdots (t-a_s )}{(t-b_1) \cdots
  (t-b_s)}.$$ Let us fix $\kappa\in (0,1).$

Throughout the proof we will use the following fact. Consider a
function
$$R(z)=\frac{(z-u_1)(z-u_2)\cdots (z-u_s)}{(z-v_1)(z-v_2)\cdots
  (z-v_s)},\quad z\in \mathbb R;$$ where $\{u_i\}$ and $\{v_i\}$ interlace. Then $R(z)$ is strictly piecewise
monotone between its singularities. It can be proved by induction using the decomposition of
$R(z)$ into the sum of simple fractions.
\begin{lem}  \label{root}
In the case of the  uniform measure on the union of the intervals
$[a_i, b_i]$ the system of equations $($\ref{sys}$)$ for $x\in \mathbb R$ has at most one pair of complex conjugate solutions.
\end{lem}
\begin{proof}

Let us express $t$ as a function of $x,$ $\kappa$ and $z$ from the first equation of (\ref{sys})
\begin{equation} \label{t}
t(z,\ka, x)=x (1-\ka )+\frac{2\ka z}{  z^2-1}
\end{equation}
 and substitute it into the second one. Let us fix $x.$ We claim that the resulting
equation (\ref{neww}) will have at most one pair of complex conjugate solutions:
\begin{equation}\label{neww}
G(z, x)=z,
\end{equation}
where
 \begin{multline}
G(z, x)=\frac{(2\ka z+x (1-\ka)(z^2-1)-a_1 (z^2-1) ) \cdots (2\ka z+x (1-\ka)(z^2-1)-a_s(z^2-1) )}{ (2\ka z+x (1-\ka)(z^2-1)-b_1(z^2-1)) \cdots (2\ka z+x (1-\ka)(z^2-1)-b_s(z^2-1))}.
\end{multline}

Let us denote the numerator of $G(z, x)$ by $G_a(z, x)$ and its denominator by $G_b(z, x).$

Let us compute the discriminant of the equation
$$z^2(1-\ka x)+2\ka z+(\ka x-x)-(z^2-1) y =0,$$
which is equivalent to $t(z, x,\kappa)=y.$

It is equal
to $$D=y^2+(2\ka x-2 x)y+(\ka^2+x^2+\ka^2 x^2-2\ka x^2)=(y+(\ka x-x))^2+\kappa^2\geq
0.$$
 Therefore, both polynomials  $G_{a}(z, x)$ and $G_b(z, x)$ have only real
 roots, moreover, since $\ka>0$ all their roots are distinct. Let
 $\{\alpha_1, \dots, \alpha_{2s}\}$ be the ordered set of roots of polynomial
 $G_{a}(z,x)$ and let  $\{\beta_1, \dots, \beta_{2s}\}$ be the ordered set of roots of polynomial
 $G_{b}(z,x).$

We claim that the roots of $G_a(z,x)$ and $G_b(z,x)$ interlace.

The function $t(z, x, \ka)$ is decreasing on the intervals
$(-\infty, -1),$ $(-1, 1)$ and $(1,\infty).$ It follows from the
computation of its derivative which is equal to $-\frac{2
  \ka (1+z^2)}{(z^2-1)^2}$ (recall $\ka\in (0,1)).$

 \begin{figure}[h]
\includegraphics[width=0.5\linewidth]{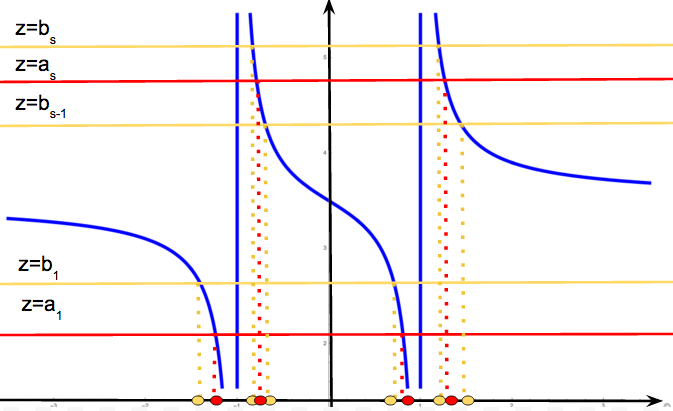}
\caption{An example of the graph of $y=t(z, x, \ka).$}
 \label{fig:grafik}
\end{figure}

Thus, since $a_i$
and $b_i$ interlace the solutions of equations $t(z, x, \ka)=a_i$ and
$t(z, x, \ka)=b_i$ will also interlace on each interval of
monotonicity. Looking at a general graph of $t(z, x, \ka)$ and examining the cases
 we will conclude that the interlacing condition is preserved on the whole line,
see Figure \ref{fig:grafik}.  Indeed, let us consider the interval $(-\infty, -1).$
Let the lines $y=a_i$ and $y=b_i$ intersect the graph of $t(z, x,
\ka)$ on this interval. Then the right-most point will be the
intersection with the line  $y=a_1.$  Consider the interval $(-1, 1).$
Then the left-most point of intersection will be the
intersection with the line  $y=b_s$ and the right-most point of intersection will be the
intersection with the line  $y=a_1.$ Finally, consider the interval $(1, \infty).$
Let the lines $y=a_i$ and $y=b_i$ intersect the graph of $t(z, x,
\ka)$ on this interval then the left-most point of intersection will be the
intersection with the line  $y=b_s.$

The leading coefficient of $G_a(z, x)$ is
$\prod\limits_{i=1}^{2s}(x(1-\ka)-a_i)$ and the leading coefficient of $G_b(z,
x)$ is $\prod\limits_{i=1}^{2s}(x(1-\ka)-b_i).$  First, assume that
$x(1-\ka)-b_i\neq 0$ for $i=1,\dots, 2s.$ Since $\{\alpha_i\}$ and $\{\beta_i\}$ interlace the function $G(z,x)$ is always piecewise monotone in $z$. Examining the cases in Figure \ref{fig:grafik} we conclude that
$$ \text{ if } \alpha_1<\beta_1, \text{ then } \frac{\prod\limits_{i=1}^{2s}(x(1-\ka)-a_i)}{\prod\limits_{i=1}^{2s}(x(1-\ka)-b_i)}<0, \text{ and }$$
$$ \text{ if } \alpha_1>\beta_1, \text{ then } \frac{\prod\limits_{i=1}^{2s}(x(1-\ka)-a_i)}{\prod\limits_{i=1}^{2s}(x(1-\ka)-b_i)}>0 .$$

Thus, $G(z, x)$ is increasing on every interval $(\beta_i, \beta_{i+1})$ from $-\infty$ to $\infty.$ There are $2s-1$ such intervals and on every interval there is at least one real solution to $($\ref{neww}$),$ see Figure \ref{fig:equation}.
Since the degree of the corresponding polynomial equation is $2s+1$
there exists at most one pair of complex conjugate roots.

 \begin{figure}[h]
\includegraphics[width=0.5\linewidth]{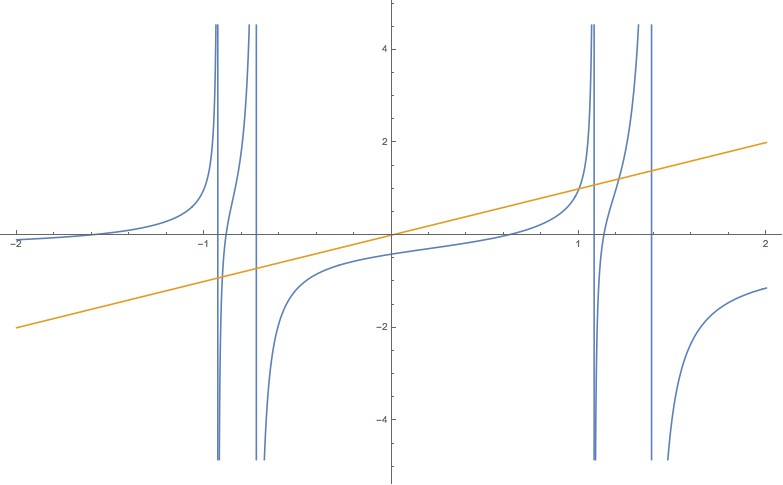}
\caption{An example of the graph of $y=G(z,x_0)$ and its intersection with $y=z.$}
\label{fig:equation}
\end{figure}
\end{proof}

Consider the case when $x(1-\ka)-b_i=0$ for some $i.$ Examining
the  Figure \ref{fig:grafik} we see that $
\alpha_1<\beta_1<\dots<\beta_{2s-1}<\alpha_s.$ Note that
$$\frac{\prod\limits_{j=1}^{2s}(b_i-a_j)}{2\ka\prod\limits_{j=1,
    j\neq i}^{2s}(b_i-b_j)}<0.$$
Similarly it can be shown that $G(z, x)$ is piecewise increasing. We again can
find $2s-2$ roots on every interval $(\beta_i, \beta_{i+1}).$ Since the degree of the corresponding polynomial equation is $2s$
there exists at most one pair of complex conjugate roots.
\end{proof}

Let $x_0$ be such that $($\ref{neww}$)$ has a pair of complex
roots. From the proof of Lemma \ref{root} we see that
$($\ref{neww}$)$ has $2s-1$ distinct real roots when $x(1-\ka)-b_i\neq
0$ for $i=1,\dots, 2s$ and $2s-2$ distinct real roots when $x(1-\ka)-b_i=
0$ for some $i.$ Due to the Implicit Function Theorem those roots are
well-defined in some complex neighborhood $\mathcal U$ of $x$. Let us
denote them $z_j(x)$ (where $j=1,\dots, 2s-1$ or $j=1,\dots, 2s-2$
corresponding to two cases mentioned above.)

\begin{lem} The derivative of $z_j(x)$ with respect to $x$ at $x_0$ is non-negative, for $i=1, \dots 2s-1.$ Moreover, it is equal to zero
if and only if $z_j(x_0)=1.$
\end{lem}
\begin{proof}
Since $z_i(x)$ is given as an implicit function, its derivative can be computed in the following way
\begin{equation}
z_i'(x)=-\frac{G'_x(z, x)}{G'_z(z, x)-1}.
\end{equation}
First, we show that $G'_x(z, x)\leq 0.$ Recall that
$$ G(z, x)=\frac{(x (1-\ka)+\frac{2\ka z}{z^2-1 }-a_1) \cdots (x (1-\ka)+\frac{2\ka z}{z^2-1 }-a_s )}{ (x (1-\ka)+\frac{2\ka z}{z^2-1 }-b_1) \cdots (x (1-\ka)+\frac{2\ka z}{z^2-1 }-b_s)}.$$
Thus, when $z\neq \pm 1,$ the statement of the lemma is equivalent to the statement that function
$$f(y)=\frac{(y-u_1)\cdots (y-u_s)}{(y-v_1)\cdots (y-v_s)}$$ is a piecewise decreasing function,
where $u_1<v_1<u_2<\dots<u_n<v_n.$ Thus, it follows that $G'_x(z, x)< 0$ for $z\neq \pm 1.$ Note that $z=-1$ can never be a solution to $($\ref{sys}$)$ and when $z=1$ we have $G'_x(z, x)=0.$

Next, we show that $G'_z(z_i(x_0), x_0)>1.$
From the proof of Lemma \ref{root} we know that $G(z, x_0)$ is an
increasing function from $-\infty$ to $\infty$ on every interval
$(\beta_i, \beta_{i+1}).$ Since $x_0$ is such that $($\ref{neww}$)$ has a pair of complex
roots we conclude that there is exactly one real
solution $z_j(x_0)$ to $($\ref{neww}$)$ on every interval
$(\beta_i, \beta_{i+1}).$
Therefore, $G'_z(z_j(x_0), x_0)>1$ since $G(z,x)-z$ goes from $-\infty$ to $\infty$ on $(\beta_i, \beta_{i+1})$ and has exactly one root.

Now let us consider $\tilde x=x_0+i\varepsilon \in \mathcal U,$ for
some very small $\varepsilon.$ From the sign of derivative we obtain
$\textup{Im}(z_i(\tilde x))>0.$ On the other hand,
$${\bold Z}^{ \kappa}(\tilde x)=\exp\left (\int\limits_{\mathbb R} \frac{((x_0-s)-i\varepsilon) m[ds]}{(x_0-s)^{2}+\varepsilon^2}\right).$$
It follows that $\textup{Im}({\bold Z}^{ \kappa}(\tilde x))<0.$
Thus, ${\bold Z}^{ \kappa}(x)$ cannot be equal to a real root if the
complex roots exist. Therefore, ${\bold Z}^{ \kappa}(x)$ coincides with the unique complex root with negative imaginary part at $x_0.$

Finally, from Lemma \ref{support} we conclude that
$$
{\bold d} \mes^\kappa(x)=-\lim\limits_{\varepsilon\rightarrow
  0^+}\frac{1}{\pi}\textup{Im}(\textup{St}_{\mes^\kappa}(x+i\varepsilon))=-\lim\limits_{\varepsilon\rightarrow
  0^+}\frac{1}{\pi}\textup{Im}\left(\log\left({\bold Z}^{ \kappa}(x+i\varepsilon)\right)\right)=\frac{1}{\pi}\textup{Arg}( {\bold
  z}^\kappa_+(x)).$$
 for this class of measures.

Now let us consider a general case of a measure $\mes_\omega.$ There
exists a sequence of measures $\{\eta_i\}$ that converges weakly to
$\mes_\omega,$ where $\eta_i$ is a uniform measure with density one on
a sequence of intervals. Then we can pass to the limit in equation \eqref{eq:density_formula}.
\end{proof}

\begin{rem}
Note that the scaling which we use computing $\mes^\kappa$ is
$\frac{1}{(1-\kappa)N},$ while the coordinates of the domain scale
as $\frac{1}{N}.$
\end{rem}

   \begin{defi} \label{liquid_region}
 Let $\mathcal L$ be the set of $(\chi, \kappa)$ inside $\mathcal R$ such that
 the density ${\bold d} \mes^\kappa(\frac{\chi}{1-\kappa})$ is not equal
 to $0$ or $1$. Then $\mathcal L$ is called the liquid region. Its boundary $\partial \mathcal L$ is called the frozen boundary.
 \end{defi}

Note that
the density ${\bold d} \mes^\kappa(\frac{\chi}{1-\kappa})$ is not equal to 0 or 1
if and only if the system (\ref{sys}) has a complex solution for $(\chi, \kappa)\in \mathcal R.$

%% file: coordinates.tex
\section{Frozen boundary}
\label{sec:frozen}
In this section we consider a special case of a rectangular Aztec diamond $\mathcal
R(N, \a, m)$ with
$$\text{ }\a=(A_1, \dots, B_1, A_2,
 \dots, B_2, \dots, A_s, \dots, B_s), \text{  where }
 \sum_{i=1}^s(B_i-A_i+1)=N, \text{ see Figure \ref{fig:boundary}.}$$

Denote the string $(A_1, A_2 \dots, A_s)$ by
 $A^{(s)}$ and the string $(B_1,B_2 \dots, B_s)$ by $B^{(s)}.$
We call such domain a {\it rectangular Aztec diamond of type
  $(N, A^{(s)}, B^{(s)})$}.

  \begin{figure}[h]
\includegraphics[width=0.35\linewidth]{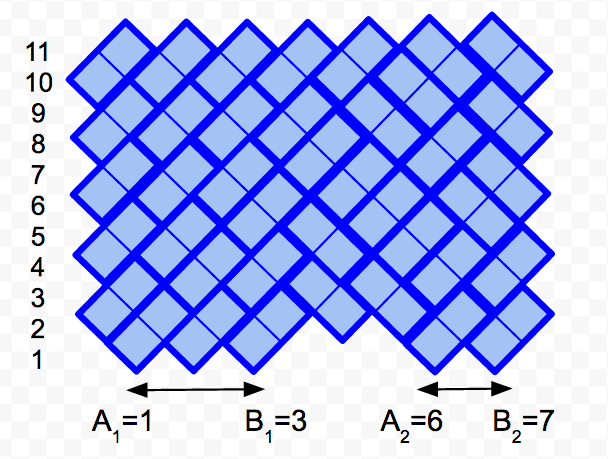}
\caption{Rectangular Aztec diamond $\mathcal
R(5, (1, 2, 3, 6, 7), 2)$ with $A_1=1,$ $B_1=3,$ $A_2=6,$ $B_2=7.$}
  \label{fig:boundary}
\end{figure}

We are interested in the following asymptotic regime of its growth :
 $$\lim_{N\rightarrow \infty} \frac{A_i(N)}{N}=a_i, \text{ }
 \lim_{N\rightarrow \infty} \frac{B_i(N)}{N}=b_i,$$
where $a_1<b_1<\dots<a_s<b_s$ are new parameters such that $\sum\limits_{i=1}^s(b_i-a_i)=1.$
We will call $s$ the number of segments.

The sequence of signatures $\{\omega(N)\}$ is regular and the limit measure $\mes_\omega$ is a uniform measure on the
union of intervals $[a_i, b_i]$ considered in the proof of the
Lemma \ref{root}.

\begin{thm} \label{frozen_b}
The frozen boundary of the limit of a rectangular Aztec diamond of type
$(N, A^{(s)}, B^{(s)})$ is a rational algebraic curve $C$ with an
explicit parametrization \eqref{eq:curveC}. Moreover, its
dual $C^{\vee}$ is of degree $2s$
and is given in the following parametric form
\begin{equation}\label{eq:curve}
C^{\vee}=\left(\theta,
\text{ }  \frac {2\theta \Pi_s(\theta)}{(\Pi_s(\theta)-1)(\Pi_s(\theta)+1)} \right),
\end{equation}
where

$$\Pi_s(\theta)=\frac{(1-a_1\theta)(1-a_2 \theta)\cdots(1-a_s \theta)}{(1-b_1 \theta)(1-b_2
  \theta)\cdots(1-b_s \theta)}.$$
\end{thm}

\begin{rem}
In projective geometry a dual curve of a given plane curve $C$ is a
curve in the dual projective plane consisting of the set of lines
tangent to $C$.

If $C$ is given in a parametric form $C=(x, y)$ then the
parametrization of its dual

$C^\vee=(x^\vee, y^\vee)$ can be found by
the following formula:
\begin{equation}\label{eq:dual_par}
\begin{cases}
x^\vee=\frac{y'}{yx'-xy'};\\
y^\vee=-\frac{x'}{yx'-xy'}.
\end{cases}
\end{equation}

We are interested in the dual curve because its parametrization can be
written in a simple form and it also encodes the information about
the tangents.

\end{rem}

\begin{proof}

In the case of a rectangular Aztec diamond of type $(N, A^{(s)}, B^{(s)})$ we can rewrite the system of equations (\ref{sys}) in the following form:
$$ \begin{dcases} \label{systema}
\frac{(t-a_1 )(t-a_2 )\cdots(t-a_s )}{(t-b_1 )(t-b_2)\cdots(t-b_s)}=z,\\
F_\kappa(z)=\frac{z}{(1 - \kappa)}\left (\frac{t}{ z} - \frac{1}{z - 1} +\frac{\kappa}{1 + z}\right)
+ \frac{z}{z - 1} =\frac{\chi}{(1-\kappa)},
\end{dcases} $$

Let us express $z$ from the second equation and plug it into the first
one. We find that
$$z^{\pm}(t, \kappa)=\frac{\kappa\pm\sqrt{\kappa^2+(t-\chi)^2}}{t-\chi}.$$

From Theorem \ref{density} it follows that $(\chi, \kappa)$ belongs to
the frozen boundary if and only if the resulting equation has a double
root.

Let us introduce the following notation $\L\Pi_s(t):=\log \Pi_s(t).$ The condition that the resulting equation has a double root can be
rewritten as the following system
$$ \begin{cases} \label{syst}
\Pi_s(t)=z^{\pm}(t, \kappa), \\
\L\Pi_s(t)'=(\log(z^{\pm}(t, \kappa))))'.
\end{cases} $$

This system can be solved for $\chi$ and $\kappa$ and we find that
\begin{equation} \label{eq:curveC}
\chi(t)=\frac{\Pi_s(t)^2\L\Pi_s(t) t+\L\Pi_s(t)t
  + \Pi_s(t)^2-1}{(\Pi_s(t)^2+1)\L\Pi_s(t)}
\text{ and }
 \kappa(t)=-\frac{(\Pi_s(\theta)^2-1)^2}{2(\Pi_s(t)^2+1)\Pi_s(t)\L\Pi_s(t)}.
\end{equation}

 \begin{figure}[h]
\includegraphics[width=0.35\linewidth]{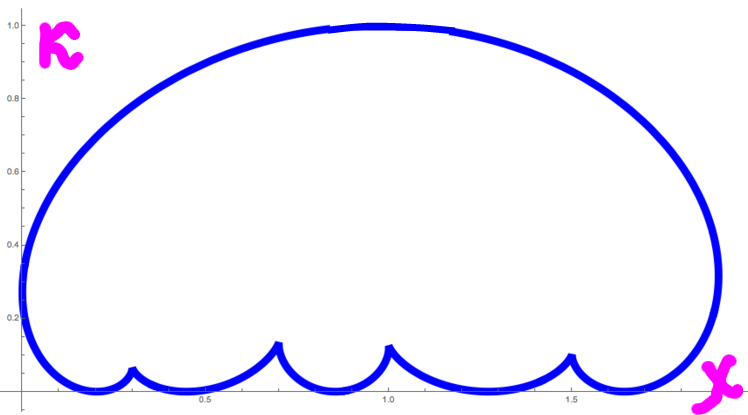}
  \label{fig:cloud}
\caption{An example of a curve $C$ with three boundary segments.}
\end{figure}

This gives us the parametric representation of the frozen boundary. In
order to compute its dual we use the formula \eqref{eq:dual_par}.

Note that $\Pi_s(t)'=\Pi_s\cdot \L\Pi_s(t).$ Using this relation we
get the parametric representation \eqref{eq:curve} of the dual curve
after some lengthy computation, where we introduced the parameter $\theta=\frac{1}{t}.$

From the parametric representation it is easy to see that the degree of the dual
curve is equal to $2s.$

\end{proof}

\begin{defi} \cite{KO} A degree $d$ real algebraic curve $C\subset
  \mathbb R\textup{P}^2$ is winding if
\begin{itemize}
\item It intersects every line $L\subset \mathbb R\textup{P}^2$ in at
  least $d-2$ points counting multiplicity;
\item There exists a point $p_0\in  \mathbb R\textup{P}^2\setminus C$
  called center, such that every line through $p_0$ intersects $C$ in
  $d$ points.

The dual to a winding curve is called a cloud curve.
\end{itemize}
\end{defi}

\begin{prop}
The  frozen boundary $C$ is a cloud curve of rank $2s,$ where $s$ is the
number of segments. Moreover, it is tangent to the following $2s+2$
lines
$$\mathcal L=\{\chi=a_i|i=1,\dots,s\}\cup \{ \chi=b_i|i=1,\dots,s\}\cup
\{\kappa=0\} \cup \{\kappa=1\}.$$
\end{prop}

\begin{proof}
We need to check that the dual curve $C^{\vee}=(x,y)$ is winding. Let us
write down the equation for the dual curve in terms of the coordinates
$x$ and $y.$
Let us denote
$$P_a(x)=(1-a_1x)(1-a_2x)\dots(1-a_sx) \text{ and }
P_b(x)=(1-b_1x)(1-b_2x)\dots (1-b_sx).$$
In terms of $x$ and $y$ we get the following equation of $C^{\vee}$
\begin{equation}
 y(P_a(x)-P_b(x))(P_a(x)+P_b(x))-2xP_a(x)P_b(x)=0
\end{equation}
Note that $P_a(0)=P_b(0)=1$. Thus, the
rank of $C$ is $2s.$

We need to compute the intersection of $C^\vee$
with any given line in $\mathbb R\textup{P}^2.$ Let us start with the lines
given by $y=cx+d,$ where $c, d\in \mathbb R.$ Then the points of intersection satisfy the
following equation
\begin{equation} \label{P}
 (cx+d)(P_a(x)-P_b(x))(P_a(x)+P_b(x))=2xP_a(x)P_b(x).
\end{equation}
We will use the same strategy as in Lemma \ref{root}. We can assume
that $a_1>0$ since we can move the limit rectangle to the right along
the $\chi$ axis,
which does not affect the geometric properties of the frozen boundary.
The roots of the RHS of equation (\ref{P}) are
$\{0, \frac{1}{a_1}, \frac{1}{b_1},\dots , \frac{1}{a_s},\dots,\frac{1}{b_s}\}.$ Since $\{a_i\}$ and $\{b_i\}$
interlace, all the roots of $p_{-}=P_a(x)-P_b(x)$ and $p_{+}=P_a(x)+P_b(x)$ are
real and there are $2s-1$ distinct roots. Moreover, there exist two roots of $p_{\pm}$ in every interval
$(\frac{1}{a_{i+1}}, \frac{1}{a_{i}})$ and the point $\frac{1}{b_i}$ lies in between those
roots. Therefore, the roots of the  polynomial given by the LHS of (\ref{P}) interlace
with the roots of the polynomial given by the RHS with probably an
exception of an interval containing the zero of $cx+d.$ So we have found
at least $2s-2$ real points of intersection of our curve with lines of
the form $y=cx+d.$

Next, consider the lines $x=d.$ The intersection will be at
infinity. We need to consider the homogenized polynomial that defines
our curve adding the third homogeneous coordinate $z$.
$$y(P_a(x, z)-P_b(x, z))(P_a(x, z)+P_b(x, z))-2xP_a(x, z)P_b(x, z)=0,$$
$$P_a(x,z)=(z-a_1x)(z-a_2x)\dots(z-a_sx) \text{ and }
P_b(x,z)=(z-b_1x)(z-b_2x)\dots (z-b_sx).$$
The homogeneous form of the equation of the line is $x=dz.$ Thus, we see
that this line intersects the curve at the point $(0:1:0)$ with
multiplicity $2s-1.$ The case of the line $z=0$ is analogous.

 Note that due to the computation above if we choose a point $p_0$ on
 the line $y=0$ outside the interval $( 0 ,\frac{1}{a_1})$
 any line passing through $p_0$ will intersect the curve $C^{\vee}$ at at least $2s-1$ real
 points. Therefore, any line through $p_0$ will intersect $C^{\vee}$ at
 exactly $2s$ real points. So the point $p_0$ can be chosen to be the
 center.

Notice that the points $(x,y)=(0, \frac{1}{a_i})$ as well as
$(x,y)=(0, \frac{1}{b_i})$ belong to $C^{\vee}.$ The roots of
$(P_a(x)-P_b(x))(P_a(x)+P_b(x))$ correspond to $2s-1$ points of
tangency of $C$ with the line $\kappa=0.$  Also, when $x=0$ the $y$
coordinate of $C^{\vee}$ is $1$, which corresponds to the tangent
$\kappa=1$ of $C.$  From these the second
statement of the proposition follows.

\end{proof}

\begin{rem} From \eqref{eq:curve} and \eqref{eq:curveC}  we see that the points of tangency
  of $C=(\chi(\theta), \kappa(\theta))$ with $\kappa=0,$ that is, with the side of rectangle with nontrivial boundary conditions, are precisely the roots of  $$\prod\limits_{i=1}^{s} \frac{x-a_i}{x-b_i}=\pm 1. $$
\end{rem}

%% file: hom.tex
\section{Central Limit Theorem}
\label{sec:6}
In this section we study the fluctuations of the limit shape. In
particular, we show that the fluctuations are described by the
Gaussian Free Field. There are two major steps in the proof of this
result. First, we describe the complex structure on the liquid region
that leads to the appearance of the Gaussian Free Field as the limit object.
Next, we give an overview of the results from \cite{BG2} that allow us
to perform the computations to establish the main result Theorem \ref{theorem:gff-domino}.

\subsection{Characterization of the liquid region}
 By Theorem \ref{density} $(\chi, \kappa) \in \mathcal L$ (see
 definition \ref{liquid_region} and the remark above it) if and only
if the system of equations (\ref{sys}) has complex roots. Let us recall it
$$ \begin{cases} \label{system}
\F_{\kappa}(z{,}t)=\frac{\chi}{(1-\kappa)}, \\
\textup{St}_{\mes_\omega}(t)=\log(z),
\end{cases} $$
where
$$\F_{\ka}(z,t)=\frac{z}{(1-\ka)}\Big (\frac{t}{z}- \frac{1}{(z-1)}+\frac{\kappa}{1 + z} \Big ) + \frac{z}{z - 1}.$$

 From the first equation we find that
  \begin{equation}\label{z}
      z_{\pm}(\chi, \kappa)(t)=\frac{\kappa
        \pm\sqrt{\kappa^2+(t-\chi)^2}}{t-\chi} \text{ and }
      t=\frac{ \chi z^2+2\ka z-\chi}{z^2-1}.
  \end{equation}

  It follows that $ z_{\pm}(\chi, \kappa)$ is complex if and only if
  $t$ is (if one of them is real then the other one has to be real too).

Let $\mathbb{H}$ be the upper half-plane of the complex plane.
 \begin{lem} \label{structure}
 Let $t\in \mathbb H.$ Then there exists $z(t)$ such that $(z(t), t)$ is the solution to $($\ref{sys}$)$ if and only if the following equation holds
 \begin{equation} \label{eq:exp_equation}
p^{\chi, \kappa}(t)\colon= \chi-\left(t+\kappa\left(\frac{1}{\exp(-\textup{St}_{\mes_\omega}(t))-1}+\frac{1}{\exp(-\textup{St}_{\mes_\omega}(t))+1}\right)\right)=0.
 \end{equation}
 \end{lem} 

 \begin{proof} Let $t$ be a solution to \eqref{eq:exp_equation}.
Let us express $\textup{St}_{\mes_\omega}(t))$ from this equation. We find that  
$$\textup{St}_{\mes_\omega}(t)=\log( z_{\pm}(\chi, \kappa)(t)).$$
 
Picking the right sign we find $z(t)$ which solves (\ref{sys}). A direct computations shows that if $(z, t)$ is a solution of \ref{sys} then $t$ is a solution to \eqref{eq:exp_equation}.

 \end{proof}
  
Note that from Lemma \ref{root} using Rouch\'e's theorem it follows that there exists a unique solution $t\in \mathbb H$ of \eqref{eq:exp_equation}. Denote
 $$ S(t):=\textup{St}_{\mes_\omega}(t)$$ for convenience.

The next Proposition defines a homeomorphism between the liquid region
and the upper half-plane $\mathbb H.$
\begin{prop}  \label{homeomorphism}

Let $$ T_\La\colon \mathcal L \rightarrow \mathbb H$$ map $(\chi, \kappa) \in
  \mathcal L$ to the corresponding unique root of \eqref{eq:exp_equation} in
  $\mathbb H.$
Then, $T_\La$ is a homeomorphism with inverse $t \rightarrow (\chi_\La(t),
\kappa_\La(t))$ for all $t\in \mathbb H$, given by
\begin{equation} \label{eq:chi}
\chi_{\mathcal L}(t)=t+\frac{\exp(S(\bar t))(\exp(2 S(t))-1)(t-\bar t)}{(\exp(S(t))-\exp(S(\bar t)))(1+\exp(S(t)+S(\bar t)))};
\end{equation}

 \begin{equation} \label{eq:kappa}
 \kappa_{\mathcal L}(t)=-\frac{(\exp(2 S(t))-1)(\exp(2S(\bar t))-1)(t-\bar t)}{2(\exp(S(t))-\exp(S(\bar t)))(1+\exp(S(t)+S(\bar t)))}.
 \end{equation}

\end{prop}
\begin{proof}
The proof is completely analogous to the proof of Theorem 2.1 in
\cite{DM}. For the reader's convenience we will present it.

We need to show
\begin{enumerate}
\item $\mathcal L$ is non-empty. More precisely, we will show that
  whenever $|t|$ is large enough $t \rightarrow
  (\chi_\La(t),\kappa_\La(t))\in \mathcal R$ and
\eqref{eq:exp_equation} holds. Then it follows that $
  (\chi_\La(t),\kappa_\La(t))\in \mathcal L$ for such $t.$
\item $\mathcal L$ is open.
\item $T_\La\colon \mathcal L\rightarrow \mathbb H$ is continuous.
\item  $T_\La\colon \mathcal L\rightarrow \mathbb H$ is injective.
\item  $T_\La\colon \mathcal L\rightarrow T_\La(\mathcal L)$ has inverse
  for all $t\in T_\La(\mathcal L).$
\item $T_\La(\mathcal L)=\mathbb H.$
\end{enumerate}

  Let us start with (1). Fix $t\in \mathbb H$ and define
\begin{equation} \label{eq:xk}
(\chi,
  \kappa)=(\chi_{\mathcal L}(t), \kappa_{\mathcal L}(t))
\end{equation}
using \eqref{eq:chi} and \eqref{eq:kappa} .
Then the following equation holds
 \begin{equation}
 \label{eq:exp-t-chi}
 (\exp(2S(t))-1)(t-\chi)=2 \exp(S(t))\kappa.
 \end{equation}
It is equivalent to \eqref{eq:exp_equation}.

 Second, let us look at the Taylor expansions for
$\chi_{\mathcal L}(t)$ and $\kappa_{\mathcal L}(t),$ when $|t|$ goes
to infinity.

Let $a, b\in \mathbb R$ be such that
$\textup{Support}(\mes_\omega)\subset [a, b].$
We have the following asymptotics for the Stieltjes transform
$$S(t)=\frac{1}{t}+\frac{\alpha}{t^2}+\frac{\beta}{t^3}+O(|t|^{-4}),$$
where $\alpha=\int_a^b x\text{ } \mes_\omega[dx]$ and $\beta=\int_a^b x^2
\mes_\omega[dx].$ After some computation we get

$$\chi=\alpha+O(|t|^{-1})  \text{  and  } \kappa=1+\left(\alpha^2-\beta-\frac{1}{6}\right)\frac{1}{|t|^2} +O(|t|^{-3}).$$

Recall that by our assumptions $\mes_\omega$ is a probability measure
such that $b-a>1$ since $\mes_\omega\leq \lambda,$ where
$\lambda$ is the Lebesgue measure. Therefore, we have
$$a+\frac{1}{2}=\int_{a}^{a+1} xdx<\alpha<\int_{b-1}^b xdx=b-\frac{1}{2}.$$
 Similarly,
$$\beta-\alpha^2>\frac{1}{2}\int\int_0^1(x-y)^2dxdy=\frac{1}{12}.$$

It follows that $(\chi, \kappa) \in(a+\frac{1}{2}, b-\frac{1}{2})\times(0,1)$
whenever $|t|$ is chosen to be sufficiently large. This proves (1).

Consider (2). Let $(\chi_1, \kappa_1)\in \mathcal L$ and
$t_1=T_{\mathcal L}(\chi_1, \kappa_1).$ We will show that
$(\chi_2,\kappa_2)\in \mathcal L$ whenever $|\chi_1-\chi_2|$ and
$|\kappa_1-\kappa_2|$ are sufficiently small. Fix $\epsilon >0$ such that $B(t_1,
\epsilon)\subset \mathbb H$. Let $t_1$ is the unique root of
$p^{\chi_1, \kappa_1}(t)$ in $\mathbb H,$ the extreme value theorem gives,
$$\inf\limits_{t\in \partial B(t_1, \epsilon)} |p^{\chi_1, \kappa_1}(t)|>0.$$

Also,
\begin{equation}
|p^{\chi_1, \kappa_1}(t)-p^{\chi_2,\kappa_2}(t)|<\epsilon
\end{equation}
for any $\epsilon>0,$ whenever $|\chi_1-\chi_2|$ and
$|\kappa_1-\kappa_2|$ are sufficiently small.

Whenever  $|\chi_1-\chi_2|$ and
$|\kappa_1-\kappa_2|$ are sufficiently small $|p^{\chi_1, \kappa_1}(t)|>|p^{\chi_1, \kappa_1}_+(t)-p^{\chi_2,
  \kappa_2}_+(t)|.$ Rouch\'es Theorem implies that $p^{\chi_1,
  \kappa_1}_+(t)$ has a root in $B(t, \epsilon).$

Consider (3). The analysis is very similar to part (2) and is
tautologically the same as in \cite{DM}.

Consider (4). Suppose there exist $(\chi_1, \kappa_1), (\chi_2, \kappa_2)\in \mathcal
L$ such that $T_{\mathcal L}(\chi_1, \kappa_1)=T_{\mathcal L}(\chi_2,
\kappa_2)=t\in \mathbb H.$

From (\ref{eq:exp_equation}) we get
$$t=\frac{\kappa_1\chi_1-\kappa_2\chi_2}{\kappa_1-\kappa_2}.$$
But then $t\in \mathbb R$ whenever $\kappa_1\neq \kappa_2,$ which
contradicts $t\in \mathbb H.$ Thus, $\kappa_1=\kappa_2\in(0,1).$ It
follows $\chi_1=\chi_2$ too.

Consider (5). This follows from \eqref{eq:exp-t-chi}.

Consider (6). We already know that $T_{\mathcal L}(\mathcal L)$ is
open and homeomorphic to $\mathcal L$. Suppose there exists $t\in \partial T_{\mathcal
  L}(\mathcal L)$ such that $t\in \mathbb H \setminus T_{\mathcal
  L}(\mathcal L).$ Let $\{t_n\}\in T_{\mathcal L}(\mathcal L)$ be a sequence
that converges to $t$ as $n\rightarrow \infty$. There exists a
subsequence of $\{(\chi_\La(t_n), \kappa_\La(t_n))\}$ that converges to
some $(\chi, \kappa).$ Then we
can pass to the limit in \eqref{eq:exp_equation} and we see that
$(\chi, \kappa)\in \mathcal R$ and
$t=T_\La(\chi, \kappa).$ This is a contradiction.
\end{proof}

%% file: fluctuations.tex
\subsection{Gaussian Free Field}
A {\it Gaussian family} is a collection of jointly Gaussian random variables $\{ \xi_a \}_{a \in \Upsilon}$
indexed by an arbitrary set $\Upsilon$. We assume that all our random variables are centered, that is
\begin{equation*}
\mathbf E \xi_u = 0, \qquad \mbox{ for all } u \in \Upsilon.
\end{equation*}
Any Gaussian family gives rise to a \emph{covariance kernel}
$\Cov : \Upsilon \times \Upsilon \to \mathbb R$ defined by
\begin{equation*}
\Cov (u_1, u_2) = \mathbf E ( \xi_{u_1} \xi_{u_2} ).
\end{equation*}

Assume that a function $\tilde C : \Upsilon \times \Upsilon \to \mathbb R$
is such that for any $n\ge 1$ and $u_1, \dots, u_n \in \Upsilon$,
$[\tilde C (u_i,u_j)]_{i,j=1}^{n}$ is a symmetric and positive-definite matrix.
Then there exists a centered Gaussian family with the covariance
kernel $\tilde C$ (see e.g. \cite{Car}).

Let $C_0^\infty$ be the space of smooth real--valued compactly
supported test functions on the upper half-plane
$\mathbb H$.
Let us set
\begin{equation*}
\tilde G (z,w):= - \frac{1}{2 \pi} \ln \left| \frac{z-w}{z - \bar w} \right|, \qquad z,w \in \mathbb H.
\end{equation*}
This is the Green function of the Laplace operator on $\mathbb H$ with Dirichlet boundary conditions.
Define a function $C : C_0^\infty \times C_0^\infty \to \mathbb R$ via
\begin{equation*}
C (f_1, f_2) := \int_{\mathbb H} \int_{\mathbb H} f_1 (z) f_2 (w) \tilde G(z,w)
dz d \bar z dw d \bar w.
\end{equation*}

The \emph{Gaussian Free Field} (GFF) $\mathfrak G$ on $\mathbb{H}$ with zero boundary conditions can be
defined as a Gaussian family $\{ \xi_f \}_{ f \in C_0^\infty}$ with covariance kernel $C$.

The integrals $\int f(z) \mathfrak G(z) dz$ over finite contours
in $\mathbb{H}$ with continuous functions $f(z)$ make sense,
cf. \cite{She}, while the field $\mathfrak G$ cannot be defined as a random function on $\mathbb H$.

\subsection{Statement of the theorem}
Consider a rectangular Aztec diamond $\mathcal R(N,\Omega,m).$ Let
$\omega$ be a signature corresponding to the boundary row under
Construction \ref{con}.

We have defined a probability measure $P^N_\omega$ on the sets of signatures
$$\mathcal S^N = (\mu^{(N)}, \nu^{(N)}, \dots, \mu^{(1)}, \nu^{(1)}),$$
see equation \eqref{eq:measure1}.

Let us define a function $\Delta^N$ on $\mathbb R_{\ge 0}\times \mathbb R_{\ge 0} \times
\mathcal S \to \mathbb N$ via
\begin{multline} \label{eq:BigDelta}
\Delta^N\colon (x, y,  (\mu^{(N)}, \nu^{(N)}, \dots, \mu^{(1)}, \nu^{(1)}))
\rightarrow  \\
\sqrt{\pi} \left| \left\{ 1 \le s \le N- \lfloor y \rfloor : \mu^{(N- \lfloor y \rfloor)}_s +(N- \lfloor y \rfloor)- s \ge x \right\}\right|.
\end{multline}

Recall the definition \eqref{eq:delta} of $\Delta^N_D(i, j),$ where
$D\in \mathfrak D(N, \Omega, m)$ and $(i,j)$ is a lattice vertex in
$\mathcal R(N,\Omega,m).$ Its domain of definition can be
extended to $(i, j)\in \mathbb R_{\ge 0}\times \mathbb R_{\ge 0}.$
We get
\begin{equation}
\Delta^N(i,j, \mu(D))= \sqrt{\pi} \Delta^N_D(i, j).
\end{equation}

Here $\mu(D)$ is the sequence of signatures corresponding to $Y(D),$
where $Y$ is a bijection between $\mathfrak D(N, \Omega, m)$ and
 $\mathcal S(N, \Omega, m),$ see Theorem \ref{bij}.

Let us denote by $\Delta^N_{\mathfrak D}(x, y)$ the pushforward of the
 measure $P^N_\omega$ on $\mathcal S^N$ with respect to
 $\Delta^N$. Note that due to Proposition \ref{uniform} $\Delta^N_{\mathfrak D}(x, y)$
 coincides with the pushforward of the uniform measure on $\mathfrak D(N, \Omega, m)$
with respect to $\Delta_D.$

Let us carry $\Delta^N _{\mathfrak D} (x,y)$ over to $\mathbb H$ in the following way
$$
{\bold \Delta} _{\mathfrak D}^{N} (z) := \Delta _{\mathfrak D}^{N} ( N \chi_{\La} (z), N \kappa_{\La} (z)), \qquad z \in \mathbb H,
$$
where $\chi_{\La} (z)$ is defined by \eqref{eq:chi}  and
$\kappa_{\La} (z)$ is defined by \eqref{eq:kappa}.
For a real number $0< \kappa < 1$ and an integer $j$ define a moment
of the random function ${\bold \Delta} _{\mathfrak D}^{N}$ as
\begin{equation}
\label{eq:before-limit-moment}
M^\kappa_j = \int_{-\infty}^{+\infty} \chi^j \left({\bold \Delta} _{\mathfrak D}^{N} ( N \chi, N \kappa) - \mathbf E {\bold \Delta} _{\mathfrak D}^{N}( N \chi, N \kappa) \right) d \chi.
\end{equation}
Also define the corresponding moment of GFF via
$$
\mathcal M^\kappa_j = \int_{z \in \mathbb H; \kappa_{\La} (z) = \kappa} \chi_{\La} (z)^j \mathfrak G(z) \frac{d \chi_{\La}(z)}{dz} dz.
$$

\begin{thm}[Central Limit Theorem]
\label{theorem:gff-domino}

Let ${\bold \Delta} _{\mathfrak D}^{N}(z)$ be a random function corresponding to the uniformly random domino tiling of the rectangular Aztec diamond in the way described above. Then
$$
{\bold \Delta} _{\mathfrak D}^{N}(z)- \mathbf E{\bold \Delta} _{\mathfrak D}^{N} (z) \xrightarrow[N \to \infty]{} \mathfrak G (z).
$$
In more details, as $N \to \infty$, the collection of random variables
$\{ M^\kappa_j \}_{1 \ge \kappa>0; j \in \mathbb{N}}$ converges, in the sense of finite-dimensional distributions, to $\{ \mathcal M^\kappa_j \}_{1 \ge \kappa>0; j \in \mathbb{N}}$.
\end{thm}
\begin{rem}
Due to Lemma \ref{height_formula} the corresponding
statement for the random height function $h_{\mathfrak D}$ follows.
\end{rem}

\begin{rem}
Theorem \ref{theorem:gff-domino} establishes the convergence for a
certain limited class of test functions. It is very plausible that the statement can be extended to a larger class of test functions, but we do not address this question here.
\end{rem}

\subsection{Schur generating functions}

For the proof of Theorem \ref{theorem:gff-domino} we will use the moment method of Schur-generating functions from \cite{BG2}.
In this section we will recall necessary notions and results from
\cite{BG2} (Theorems \ref{th:general-for-domino} and
\ref{Theorem_character_asymtotics_2} below). After this we will apply
them to our measure (Theorem \ref{theorem:covariance-domino}).

\begin{defi}
We say that a sequence of probability measures $\rho(N)$ on $\GT_N$ is \textbf{CLT-appropriate} if theirs Schur generating functions $S^{U(N)}_{\rho(N)} (x_1, \dots, x_N)$ satisfy the following condition: For any fixed $k \in \mathbb N$ we have
\begin{itemize}
\item
$$
\lim_{N \to \infty} \frac{ \pa_i \ln S^{U(N)}_{\rho(N)} (x_1, x_2, \dots, x_k, 1^{N-k} ) }{N} = F (x_i), \qquad 1 \le i \le k,
$$

\item
$$
\lim_{N \to \infty} \pa_i \pa_j \ln S^{U(N)}_{\rho(N)} (x_1, x_2, \dots, x_k, 1^{N-k}) = G (x_i, x_j), \qquad 1 \le i, j \le k, \ i \ne j,
$$
where the functions $F(x)$ and $G(x,y)$ are holomorphic in a
neighborhood of the unity, and the convergence is uniform over an open
complex neighborhood of $(x_1, \dots, x_k) = (1^k)$. Note that the
functions in the left-hand side of these equations depend on $k$
variables, while the the functions in the right-hand side depend on 1 or 2 variables, respectively.
\end{itemize}
\end{defi}

For $\la \in \GT_{k_1}, \mu \in \GT_{k_2}$, $k_1 \ge k_2$, let us
introduce the coefficients $\pr_{k_1 \to k_2} (\la \to \mu)$ (the
construction is similar to the one from Section \ref{sec:2}) via
$$
\frac{s_{\lambda} (x_1, \dots, x_{k_2}, 1^{k_1-k_2})}{s_{\la} (1^{k_1})} = \sum_{\mu \in \GT_{k_2}} \mathrm{pr}_{k_1 \to k_2} (\la \to \mu) \frac{s_{\mu} (x_1, \dots, x_{k_2})}{s_{\mu} (1^{k_2})}.
$$

For a symmetric function $g (x_1, \dots, x_N)$ we define the coefficients $\st_{(g)}^{(N)} (\la \to \mu)$, for $\la \in \GT_N$, $\mu \in \GT_N$, via
$$
g (x_1, \dots, x_N) \frac{s_{\la} (x_1, \dots, x_N)}{s_{\lambda} (1^N)} = \sum_{\mu \in \GT_N} \st_{(g)}^{(N)} (\la \to \mu) \frac{s_{\mu} (x_1, \dots, x_N)}{s_{\mu} (1^N)},
$$
where we assume that the function $g (x_1, \dots, x_N)$ is such that the number of terms in the summation in the right-hand side is finite for all $\la$.

Let $m$ be a positive integer and let $0 < a_1 \le \dots \le a_m = 1$ be reals. Let $\rho_N = \rho_{[a_m N]}$ be a CLT-appropriate measure on $\GT_{[a_m N]}$ with the Schur generating function $S_N$, and let $g_1 (x_1, \dots, x_{[a_1 N]})$, $\dots$, $g_m (x_1, \dots, x_{[a_m N]})$ be a collection of Schur generating functions of some CLT-appropriate probability measures.

Let us define the probability measure on the set
$$
\GT_{[a_m N]} \times \GT_{[a_m N]} \times \GT_{[a_{m-1} N]} \times \GT_{[a_{m-1} N]} \times \dots \times \GT_{[a_1 N]} \times \GT_{[a_1 N]}.
$$

That is, we need to define the probability of a collection of signatures $(\mu^{(m)}$, $\nu^{(m)}$, $\mu^{(m-1)}$, $\nu^{(m-1)}$, $\dots$, $\mu^{(1)}$, $\nu^{(1)})$, where $\mu^{(i)}$, $\nu^{(i)}$ are signatures of length $[a_{i} N]$, $i=1,2, \dots, m$. Let us do this in the following way.

Let
\begin{align*}
H_m^{\mu} (x_1, \dots, x_{[a_m N]}) := S_N (x_1, \dots, x_{[a_m N]}), \\
H_m^{\nu} (x_1, \dots, x_{[a_m N]}) := H_m^{\mu} (x_1, \dots, x_{[a_m N]}) g_m (x_1, \dots, x_{[a_m N]}), \\
H_{m-1}^{\mu} (x_1, \dots, x_{[a_{m-1} N]}) := H_{m}^{\nu} (x_1, \dots, x_{[a_{m-1} N]}, 1^{[a_m N]-[a_{m-1} N]}), \\
H_{m-1}^{\nu} (x_1, \dots, x_{[a_{m-1} N]}) := H_{m-1}^{\mu} (x_1, \dots, x_{[a_{m-1} N]}) g_{m-1} (x_1, \dots, x_{[a_{m-1} N]}), \\
\dots \dots \dots \dots \dots \\
H_1^{\mu} (x_1, \dots, x_{[a_1 N]}) := H_{2}^{\nu} (x_1, \dots, x_{[a_2 N]}, 1^{[a_{2} N]-[a_1 N]}), \\
H_1^{\nu} (x_1, \dots, x_{[a_1 N]}) := H_1^{\mu} (x_1, \dots, x_{[a_1 N]}) g_1 (x_1, \dots, x_{[a_1 N]}),
\end{align*}

Assume that the coefficients $\st^{[a_t N]}_{g_t} (\mu \to \nu)$, $t=1, \dots, s$ are positive for all $t$, $\mu \in \GT_{[a_t N]} $, $\nu \in \GT_{[a_t N]}$ (again, we also assume that the linear decomposition which defines these coefficients is finite).

We define the probability of the configuration $(\mu^{(m)}, \nu^{(m)}, \dots, \mu^{(1)}, \nu^{(1)})$ by
\begin{multline}
\label{eq:mes-time-space}
\mathcal{P^N} (\mu^{(m)}, \nu^{(m)}, \mu^{(m-1)}, \nu^{(m-1)}, \dots, \mu^{(1)}, \nu^{(1)}) \\ := \rho_N (\mu^{(m)}) \st^{[a_1 N]}_{g_1} ( \mu^{(1)} \to \nu^{(1)}) \prod_{i=2}^{m} \st^{[a_i N]}_{g_i} ( \mu^{(i)} \to \nu^{(i)}) \pr_{[a_i N] \to [a_{i-1} N]} (\nu^{(i)} \to \mu^{(i-1)}).
\end{multline}

The conditions above guarantee the existence of the following limits:
$$
\lim_{N \to \infty} \frac{ \pa_i \ln H^{\mu}_t (x_1, x_2, \dots, x_k, 1^{N-k} ) }{N} =: F_t (x_i), \qquad 1 \le i \le k
$$
$$
\lim_{N \to \infty} \pa_i \pa_j \ln H^{\mu}_t (x_1, x_2, \dots, x_k, 1^{N-k}) =: G_t (x_i, x_j), \qquad 1 \le i, j \le k, \ i \ne j,
$$
for any $t=1,2, \dots, m$.

Let
$$
p_{k;t} := \sum_{i=1}^{[a_{t} N]} \left( \mu^{(t)}_{i} + [a_{t} N] - i \right)^k,
$$
be (shifted) moments of these signatures.
These functions become random when we consider the random sequence
$(\mu^{(m)}, \nu^{(m)}, \dots, \mu^{(1)}, \nu^{(1)})$ distributed
according to the probability measure $\mathcal P^N$ . We are interested in the asymptotic behavior of these functions.

\begin{thm}$($\cite{BG2}$)$
\label{th:general-for-domino}
In the notations above, the collection of random functions
$$
\{ N^{-k} ( p_{k;t} - \E p_{k;t} ) \}_{t=1, \dots, m; k \in \mathbb{N} }
$$
is asymptotically Gaussian with the limit covariance
\begin{multline*}
\lim_{N \to \infty} N^{-k_1-k_2} \cov \left( p_{k_1;t_1}, p_{k_2;t_2} \right) = \frac{a_1^{k_1} a_2^{k_2} }{(2 \pi \ii)^2} \oint_{|z|=\ep} \oint_{|w| = 2 \ep} \left( \frac{1}{z} +1 + (1+z) F_{t_1} (1+z) \right)^{k_1} \\ \times \left( \frac{1}{w} +1 + (1+w) F_{t_2} (1+w) \right)^{k_2} \left( G_{t_2} (z, w) + \frac{1}{(z-w)^2} \right) dz dw,
\end{multline*}
where $\ep \ll 1$ and $1 \le t_1 \le t_2 \le m$.
\end{thm}

Now we will apply this theorem to our case. Consider $N\rightarrow \infty$ asymptotics of a rectangular Aztec diamond $\mathcal R(N, \a(N), m(N)).$ We assume that the sequence of signatures $\omega(N)$ corresponding to the
 first row is regular and the sequence of measures $\{
 m[\omega(N)]\}$ weakly converge to $\mes_\omega.$ Our probability measure on the set of domino tilings $\mathfrak D(N, \Omega,m)$ defined by \eqref{eq:measure} has $\rho$ concentrated on a single signature (but with a non-trivial asymptotic behavior), and the functions $g_i$ are equal to the powers of $\prod_i (1+x_i)$ times a constant.

Let us recall some results from \cite{BG2}. We will need the following theorem.

\begin{thm}[\cite{GP}, \cite{BG2}]
\label{Theorem_character_asymtotics_2}

Suppose that $\lambda(N)\in \GT_N$, $N=1,2,\dots$ is a regular sequence of signatures such that
 $$
  \lim_{N\to\infty} m[\lambda(N)]=\mes.
 $$
 Then we have
\begin{multline}
 \label{eq_limit_of_logarithm_2}
  \lim_{N \to \infty}
  \pa_{1} \pa_{2} \log \left(\frac{s_{\lambda(N)}(x_1, x_2, \dots, x_k, 1^{N-k})}{s_{\lambda(N)}(1^N)} \right) \\ =
   \pa_{1} \pa_{2} \log \left( 1 - (x_1-1) (x_2-1) \frac{x_1 H'_\mes(x_1) - x_2 H'_{\mes} (x_2)}{x_1-x_2} \right),
 \end{multline}
where the convergence is uniform over an open complex neighborhood of $(x_1,\dots, x_k)=(1^k)$.
\end{thm}

Construction \ref{con} assigns to each row with number $[2\kappa N]$  of $\mathcal R(N,
\Omega,m)$ a Young diagram given by a $[(1-\kappa)N]$-tuple $\mu^{([(1-\kappa) N])}_{i}=(\mu_1, \dots, \mu_{[(1-\kappa)N]}).$
 We define the moment function as
$$
p_j^{\kappa} := \sum_{i=1}^{[(1-\kappa)N]} \left( \mu^{([(1-\kappa) N])}_{i} + [(1-\kappa) N] - i \right)^j.
$$

\begin{thm}
\label{theorem:covariance-domino}
In the notations above, the collection of random functions
$$
\{ N^{-k} ( p_j^{\kappa} - \E p_j^{\kappa}  ) \}_{0 < \kappa \le 1; j \in \mathbb{N} }
$$
is asymptotically Gaussian with the limit covariance
\begin{multline} \label{eq:covariance-domino}
\lim_{N \to \infty} N^{-j_2-j_1} \cov \left( p_{j_1}^{\kappa_1} , p_{j_2}^{\kappa_2}\right) = \frac{(1-\kappa_1)^{j_1}(1-\kappa_2)^{j_2} }{(2 \pi \ii)^2} \oint_{|z|=\ep} \oint_{|w| = 2 \ep} \left( \frac{1}{z} +1 + (1+z) F_1 (1+z) \right)^{j_1} \\ \times \left( \frac{1}{w} +1 + (1+w) F_2 (1+w) \right)^{j_2} Q (z, w) dz dw,
\end{multline}

where $\ep \ll 1$ and $1 \ge \kappa_1 \ge \kappa_2 >0$,
$$
F_1 (z) = \frac{1}{1-\kappa_1} H'_{\mes_\w} (1+z) + \frac{\kappa_1}{(1-\kappa_1) (z+2) },
$$
$$
F_2 (w) = \frac{1}{1-\kappa_2} H'_{\mes_\w} (1+w) + \frac{\kappa_2}{(1-\kappa_2) (w+2) }
$$
and
$$
Q(z,w) = \pa_z \pa_w \left( \log \left( 1 - z w \frac{ (1+z) H'_{\mes_\w} (1+z) - (1+w) H'_{\mes_\w} (1+w)}{z-w} \right) \right)+\frac{1}{(z-w)^2}.
$$
\end{thm}

\begin{proof}

It immediately follows from Theorem \ref{th:general-for-domino},
Theorem \ref{Theorem_character_asymtotics_2}  and Lemma \ref{Schur_generating}.

\end{proof}

\begin{rem}
\label{rem:clt}

Let us make several comments about this result.

The general mechanism of Theorem \ref{th:general-for-domino}
automatically produces the central limit theorem for moments. However,
a further analysis is necessary in order to show that the obtained covariance matches the one coming from the Gaussian Free Field. 

One can use another generator of the semigroup with respect to
the quantized free convolution $($see Remark \ref{rem:lln}$)$.  In the
case of rectangular Aztec diamonds we used
$$\prod\limits_i (1+x_i)/2 = \prod\limits_i (1+ \frac{1}{2} (x_i-1)),$$ which can be
understood as an \textit{extreme beta character} with parameter
$\frac{1}{2}$. Instead, one can use formulas $$\prod (1 + \beta
(x_i-1)) \text{ or } \prod (1 + \alpha (x_i-1))^{-1}, \text{ where }
0< \beta<1, \text{ and } 0<\alpha<
+\infty. $$ One can immediately generalize our results to the case
when the probability measure comes from these formulas. We give some details for $0<\beta<1$ case in \nameref{sec:A}.

One can also analyze another types of processes, in particular, Schur processes $($or ensembles of non-intersecting paths, see Section \ref{sec:paths}$)$. In this case one needs to suitably modify Theorem \ref{th:general-for-domino}; these generalizations quite straightforwardly follow from the technique of \cite{BG2}. After this, one can apply the technique of this paper in order to extract all necessary information about limit shapes, frozen boundary, and global fluctuations.

Finally, let us remark that while the moment method proved to be very
convenient for the study of the global behavior of our model, it does
not give insight on the local behavior. Nevertheless, it is possible to
extract some information about it from the known
results on lozenge tilings. We discuss it in \nameref{sec:B}.
\end{rem}

\subsection{Proof of Theorem \ref{theorem:gff-domino} }

The goal of this section is to obtain Theorem \ref{theorem:gff-domino} from Theorem \ref{theorem:covariance-domino}.

First, let us make a change of variables
$$\tilde z =S_{\mes_\w}^{(-1)} \left( \log (1+z) \right), \qquad \tilde w =S_{\mes_\w}^{(-1)} \left( \log (1+w) \right),$$
in equation \eqref{eq:covariance-domino}.

Using the connection between $S_{\mes_\w} (z)$ and $
H'_{\mes_\w} (z)$ (see \eqref{eq_H_derivative}), we have
\begin{align*}
\frac{1}{z} + 1 + \frac{(1+z) H'_{\mes_\w} (z+1)}{1-\kappa_1} = \frac{1}{1-\kappa_1} \left( \frac{1}{\tilde z} + \frac{\kappa_1}{\exp( -S_{\mes_\w} (\tilde z)) - 1} \right), \\
\frac{1}{w} + 1 + \frac{(1+w) H'_{\mes_\w} (w+1)}{1-\kappa_2} = \frac{1}{1-\kappa_2} \left( \frac{1}{\tilde w} + \frac{\kappa_2}{\exp( -S_{\mes_\w} (\tilde w)) - 1} \right), \\
\log \left( \left( \frac{1}{w} +1 + (1+w) H'_{\mes_\w} (w) \right) - \left( \frac{1}{z} +1 + (1+z) H'_{\mes_\w} (z) \right) \right) = \log \left( \frac{1}{\tilde w} - \frac{1}{\tilde z} \right).
\end{align*}
Substituting these equalities and slightly transforming the expression, we obtain that the right-hand side of \eqref{eq:covariance-domino} is equal to
\begin{multline}
\label{eq:gff-covariance-nice-form}
\frac{1}{(2 \pi \ii)^2 } \oint_{|z|=\ep} \oint_{|w| = 2 \ep} \left( \frac{1}{\tilde z} + \frac{\kappa_1}{\exp(-S_{\mes_\w} (\tilde z)) - 1} + \frac{\kappa_1}{\exp( -S_{\mes_\w} (\tilde z))+1} \right)^{j_1} \\ \times \left( \frac{1}{\tilde w} + \frac{\kappa_2}{\exp(-S_{\mes_\w} (\tilde w)) - 1} + \frac{\kappa_2}{\exp( -S_{\mes_\w} (\tilde w))+1} \right)^{j_2} \frac{1}{(\tilde z- \tilde w)^2} d \tilde z d \tilde w.
\end{multline}

For any $0< \kappa \le 1$ the set $\{ z \in \mathbb H: \kappa_{\La} (z) = \kappa \}$ is a well-defined curve in $\mathbb H$. Let $\mathcal Z_{\kappa}$ be a union of this curve and its complex conjugate.

Let $0 < \kappa_1 \le \kappa_2 \le 1$. Given the explicit formulas \eqref{eq:chi} and \eqref{eq:kappa}, it is easy to check that the curve $\mathcal Z_{\kappa_2}$ encircles the curve $\mathcal Z_{\kappa_1}$. Therefore, we can deform contours in \eqref{eq:gff-covariance-nice-form} and obtain
\begin{multline}
\label{eq:gff-2}
\frac{1}{(2 \pi \ii)^2 } \oint_{\tilde z \in \mathcal Z (\kappa_1) } \oint_{ \tilde w \in \mathcal Z(\kappa_2) } \left( \frac{1}{\tilde z} + \frac{ \kappa_1}{\exp(- S_{\mes} (\tilde z)) - 1} + \frac{\kappa_1}{\exp( -S_{\mes_\w} (\tilde z))+1} \right)^{j_1} \\ \times \left( \frac{1}{\tilde w} + \frac{\kappa_2}{\exp(- S_{\mes_\w}(\tilde w)) - 1} + \frac{\kappa_2}{\exp( -S_{\mes_\w} (\tilde w))+1} \right)^{j_2} \frac{1}{(\tilde z- \tilde w)^2} d \tilde z d \tilde w.
\end{multline}
Note that for $\tilde z \in \mathcal Z (\kappa_1)$ the expression
$$
\frac{1}{\tilde z} + \frac{\kappa_1}{\exp(-S_{\mes_\w} (\tilde z)) - 1} + \frac{\kappa_1}{\exp( -S_{\mes_\w}(\tilde z))+1}
$$
is real and is equal to $\chi_{\La}(z),$ see Lemma \ref{structure}. Therefore, we can rewrite \eqref{eq:gff-2} as
\begin{equation*}
-\frac{1}{4 \pi^2 } \oint_{\tilde z \in \mathcal Z (\kappa_1) } \oint_{ \tilde w \in \mathcal Z (\kappa_2) } \chi_{\La} (\tilde z)^{j_1} \chi_{\La} (\tilde w)^{j_2} \frac{1}{(\tilde z- \tilde w)^2} d \tilde z d \tilde w,
\end{equation*}

Integrating \eqref{eq:before-limit-moment} by parts we see that
$$
M^\kappa_ j = \frac{N^{-j+1} \sqrt{\pi}}{j+1} \left( p_{j+1}^\kappa - \E p_{j+1}^ \kappa \right).
$$
Therefore, the set $\{ M^\kappa_ j \}_{1 \ge \kappa >0, j \in \mathbb{N}}$ converges to the Gaussian distribution with zero mean and limit covariance
\begin{multline}
\label{eq:limit-final}
\lim_{N \to \infty} \cov \left( M^{\kappa_1}_ {j_1}, M^{\kappa_1}_ {j_1} \right) = \frac{-1}{4 \pi (j_1+1) (j_2+1) } \oint_{z \in \mathcal Z (\kappa_1)} \oint_{z \in \mathcal Z (\kappa_2)} \chi_{\La} (z)^{j_1} \chi_{\La} (z)^{j_2} \\ \times \frac{d \chi_{\La}(z)}{dz} \frac{d \chi_{\La}(w)}{dw} \frac{1}{(z-w)^2} dz dw.
\end{multline}
By definition, the set $\{ \mathcal M^{\kappa}_ {j} \}_{1 \ge \kappa
  >0, j \in \mathbb{N}}$ is Gaussian with zero mean and covariance
\begin{multline}
\cov \left( \mathcal M^{\kappa_1}_ {j_1}, \mathcal M^{\kappa_2}_ {j_2} \right) = \oint_{z \in \mathbb H: \kappa_{\La} (z) = \kappa_1} \oint_{z \in \mathbb H: \kappa_{\La} (z) = \kappa_2} \chi_{\La} (z)^{j_1} \chi_{\La} (z)^{j_2} \\ \times \frac{d \chi_{\La}(z)}{dz} \frac{d \chi_{\La}(w)}{dw} \left( \frac{-1}{2 \pi} \ln \left| \frac{z- w}{z - \bar{w}} \right| \right) dz dw.
\end{multline}
Using the equality
$$
2 \ln \left| \frac{z- w}{z - \bar{w}} \right| = \ln (z -w) - \ln (z - \bar{w}) - \ln ( \bar{z} - w) + \ln ( \bar{z} - \bar{w}),
$$
we can write it in the form
$$
\cov \left( \mathcal M^{\kappa_1}_ {j_1}, \mathcal M^{\kappa_2}_ {j_2} \right) = -\frac{1}{4 \pi} \oint_{z \in \mathcal Z (\kappa_1)} \oint_{z \in \mathcal Z (\kappa_2)} \chi_{\mes} (z)^{j_1} \chi_{\mes} (z)^{j_2} \frac{d \chi_{\mes}(z)}{dz} \frac{d \chi_{\mes}(w)}{dw} \ln (z-w) dz dw.
$$
Integration by parts shows that this formula coincides with the right-hand side of \eqref{eq:limit-final} which concludes the proof of the theorem. 

%% file: examples.tex
\section{Examples of the frozen boundary}
\label{sec:examples}

\subsection{The Aztec diamond}
\label{sec:71}
 The Aztec diamond is a rectangular
Aztec diamond of type $(N, A^1=1, B^1=N) .$  In this case we have
$a_1=0$ and $b_1=1.$

Therefore, we obtain
$$\textup{St}_{\mes_\w}(t)= \log \frac{t}{
  t-1}.$$

Then we solve the equation $\textup{St}_{\mes_\w}(t)=\log z$ for
$t$ and substitute it into \eqref{eq:F} we get
$$\F_{\kappa}(z)=\frac{z}{z-1}+\frac{\kappa z}{(1-\kappa)(z+1)}.$$

Solving explicitly the
equation $$\F_{\kappa}(z)=\frac{\chi}{1-\kappa}$$ we find
$${\bold z}^\kappa_{+}(\chi)=\frac{-1+2\kappa+\sqrt{(2 \chi-1)^2+(2
    \kappa -1)^2-1}}{2(1-\chi)}.
$$

Thus, the density $${\bold d} \mes^\kappa\left(\frac{\chi}{1-\kappa}\right)=\frac{1}{\pi}\textup{Arg}\left(\frac{1-2\kappa-\sqrt{(2 \chi-1)^2+(2
    \kappa -1)^2-1}}{2(\chi-1)}\right),$$
where $(\chi, \kappa)$ is such that $(2 \chi-1)^2+(2
    \kappa -1)^2-1<0.$

Thus, the frozen boundary is given by the equation $$(2 \chi-1)^2+(2
    \kappa -1)^2-1=0,$$

as expected.

 \begin{figure}[h]
\includegraphics[width=0.2\linewidth]{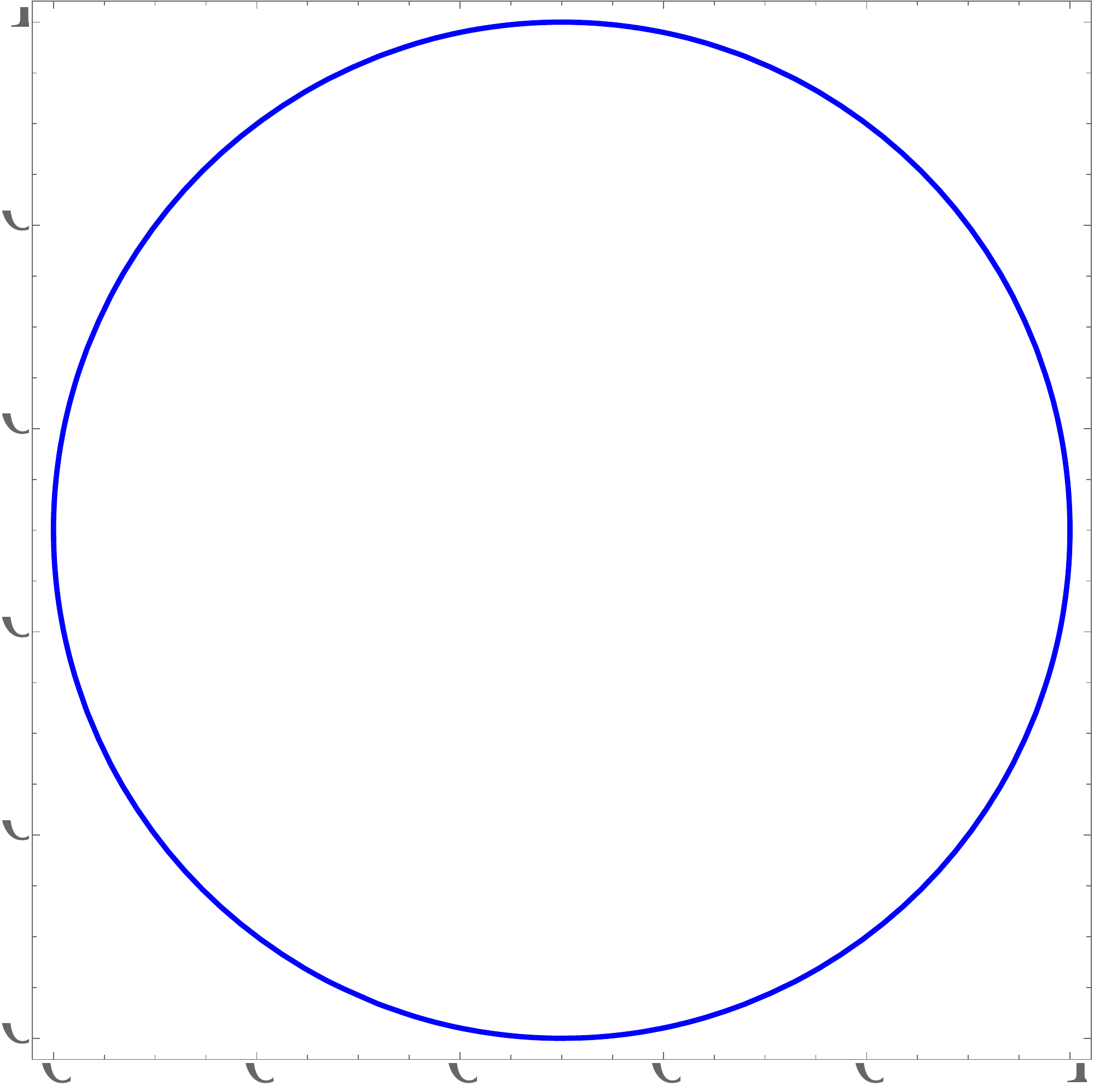}
\caption{The Arctic circle.}
  \label{fig:krug}
\end{figure}

\subsection{The Aztec half-diamond}

\begin{defi} The Aztec half-diamond is a rectangular Aztec diamond
  $\mathcal R(N, \a, N-1),$ where $\a=(1, 3, \dots, 2
  N-1).$
\end{defi}

This domain was considered in \cite{FF} and subsequently in \cite{NY}.

In this case we obtain that $\mes_{\w}$ is a uniform measure on $[0, 2].$
We can explicitly compute
$$\textup{St}_{\mes_{\omega}}(t)=-\frac{1}{2}\log\left( 1-\frac{2}{
    t}\right ).$$
Then we solve the equation $\textup{St}_{\mes_\w}(t)=\log z$ for
$t$ and substitute it into \eqref{eq:F} we get
$$\F_\kappa(z)=\frac{2 z( \kappa -  z )}{(-1 + z) (1 + z)}=\chi.$$
Therefore,
$${\bold d} \mes^\kappa\left(\frac{\chi}{1-\kappa}\right)=\frac{1}{\pi}\textup{Arg} \left(\frac{-\kappa-\sqrt{\kappa^2+\chi(
    \chi-2)}}{\chi-2}\right).
$$
So the frozen boundary is the curve given by the equation
$$\kappa^2+\chi(
    \chi-2)=0$$
for $\kappa\in [0, 1)$ and $\chi\in [0, 2].$

  \begin{figure}[h]
\includegraphics[width=0.3\linewidth]{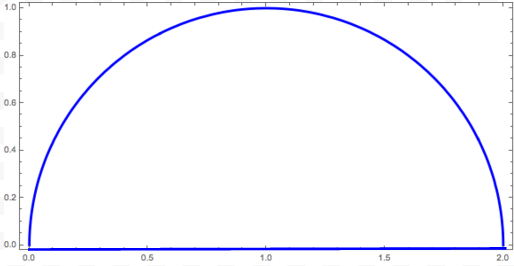}
\caption{The frozen boundary for a uniform measure on $[0,2].$}
  \label{fig:parabola}
\end{figure}

\subsection{The case of the uniform measure on $[0, \theta]$}
\label{sec:73}
We consider a rectangular Aztec diamond
  $\mathcal R(N, \a, (\theta-1) (N-1)),$ where $\a=(1, 1+\theta, \dots, \theta
  (N-1)+1)$ and $\theta \in \mathbb Z_{>0}.$

We obtain that $\mes_{\w}$ is a uniform measure on $[0, \theta].$
In this case we can explicitly compute
$$\textup{St}_{\mes_{\omega}}(t)=-\frac{1}{\theta}\log\left(1-\frac{\theta}{t}\right).$$

Let $\theta=4.$
We can explicitly solve the system (\ref{system}) on the computer
and we obtain that the frozen boundary is the curve given by the equation
$$f(\chi,\kappa)=27\kappa^4+(\kappa^2+(\chi-4)\chi)^3.$$
for $\kappa\in [0, 1)$ and $\chi\in [0, 4].$

  \begin{figure}[h]
\includegraphics[width=0.4\linewidth]{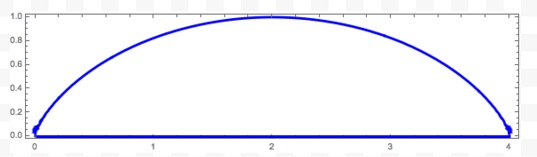}
\caption{The frozen boundary for a uniform measure on $[0,4].$}
  \label{fig:nonparabola}
\end{figure}

\subsection{Frozen boundary for general boundary conditions} For an
arbitrary measure $\mes_w$ the limit shape (and, in particular, the
frozen boundary) is determined by the quantized free projection of
$\mes_w$, see Remark \ref{rem:lln}. However, it is not easy to obtain an
explicit parametrization of the frozen boundary from such a description. We
performed a necessary analysis in some particular cases in Sections \ref{sec:frozen}
and Sections \ref{sec:71}--\ref{sec:73}, and we do not address the general case in details.

Proposition \ref{homeomorphism} gives an explicit connection between
the geometry of the liquid region in our model and the model of random
lozenge tilings of ``half-hexagons'' studied by Duse and Metcalfe
\cite{DM} and \cite{DM1}.  Moreover, the techniques developed by these authors can be directly applied in our situation. Let us mention one
corollary of the results from \cite{DM}, \cite{DM1}.

For a general limit measure the frozen boundary can have a complicated structure (see
the discussion and examples in \cite{DM1}), however, it is always possible to give an explicit
parametrization (not necessarily algebraic) of a {\it part} of
the boundary called the {\it edge}, a natural boundary on which universal asymptotic
behavior is expected. 

For a set $S\subset \mathbb R$ let $\bar S$ denote its
closure and $S^\circ$ denote its interior.
From Lemma 2.2 in
\cite{DM} it follows that $$(\chi_\mathcal L(\cdot),
\kappa_\mathcal L(\cdot))\colon \mathbb H \rightarrow \mathcal L$$
defined by \eqref{eq:chi} and \eqref{eq:kappa}  has a unique
continuous extension to an open set $R\subset \mathbb R$ given by
\begin{equation} \label{eq:R}
 R\colon =\left( \overline{ (\mathbb R \backslash
    \textup{Supp}(\mes_\omega)) \cup (\mathbb R \backslash
    \textup{Supp}(\lambda -\mes_\omega))}\right)^{\circ},
\end{equation}
where $\lambda$ denotes the Lebesgue measure (recall that under our usual
assumptions $\mes_\omega$ is absolutely continuous with respect
to Lebesgue measure and has density $\leq 1$). 

\begin{defi} The edge $\mathcal E$ is a smooth curve which is the image of the extended map
 $(\chi_\mathcal L(\cdot),
\kappa_\mathcal L(\cdot))\colon R\rightarrow \partial \mathcal
L. $
\end{defi}

The formulas \eqref{eq:chi} and \eqref{eq:kappa} for $(\chi_\mathcal L(\cdot),
\kappa_\mathcal L(\cdot))$ give rise to an explicit parametrization of
$\mathcal E.$ It is natural to conjecture that
in a generic situation the edge is not an algebraic curve (its
parametrization involves exponents of the Stieltjes transform of the
limit measure), but we do not address this question formally. In the case considered in Section \ref{sec:frozen} when
$\mes_{\omega}$ is a uniform measure on a union of segments $R=\mathbb
R$ and the parametrization is the same as the one obtained in Theorem \ref{frozen_b}.

%% file: appendixA.tex
\section{Appendix A}
\label{sec:A}
Let $0<\beta<1.$
Consider the following more general measure on the set of tilings of
rectangular Aztec diamond $\mathcal R(\a,N,m).$  In the case of the
Aztec diamond it was discussed in \cite{CJY}.

Let $\mu^{(n)}$ and $\nu^{(n)}$ be two non-negative signatures of length
$n.$ Define the coefficients
$\st_{\b}(\mu^{(n)} \rightarrow \nu_{\b}^{(n)})$ and $\pr(\nu^{(n)} \rightarrow \mu^{(n-1)})$ via
\begin{equation}  \label{sb}
\frac{s_{\mu^{(n)}}(u_1,\dots,u_n)}{s_{\mu^{(n)}}(1^n)}\prod\limits^{n}_{i=1}(1+\b(u_i-1))=\sum_{\nu^{(n)}\in\GT_n}\st_\b(\mu^{(n)} \rightarrow \nu^{(n)}) \frac{s_{\nu^{(n)}}(u_1,\dots,u_n)}{s_{\nu^{(n)}}(1^n)},
\end{equation}
\begin{equation} \label{pb}
\frac{s_{\nu^{(n)}}(u_1,\dots,u_{n-1},1)}{s_{\nu^{(n)}}(1^n)}=\sum_{\mu^{(n-1)}\in\GT_{n-1}}\pr_\b(\nu^{(n)} \rightarrow \mu^{(n-1)}) \frac{s_{\mu^{(n-1)}}(u_1,\dots,u_n)}{s_{\mu^{(n-1)}}(1^n)},
\end{equation}

Analogously to the construction from Section \ref{sec:2} we can define a probability measure on the sequences of signatures of the
form
$$\mathcal S^N=\{(\mu^{(N)}, \nu^{(N)},\dots,\mu^{(1)}, \nu^{(1)})\}$$ by the formula
\begin{multline} \label{eq:measure}
\mathcal P^{N,\b}_{\mu}((\mu^{(N)}, \nu^{(N)},\dots,\mu^{(1)}, \nu^{(1)}))=\\
=1_{\mu^{(N)}=\mu}\st_\b(\mu^{(N)} \rightarrow \nu^{(N)}) \prod\limits_{j=1}^{N-1}
( \pr_\b(\nu^{(N-j+1)} \rightarrow \mu^{(N-j)})\st_\b(\mu^{(N-j)} \rightarrow \nu^{(N-j)})),
\end{multline}
where $\mu^{(i)},\nu^{(i)} \in \GT_i. $

Let $\mathcal R(\a, N,m)$ be a rectangular Aztec
  diamond and $\omega$ be a
  signature corresponding to its boundary row. Then the measure $\mathcal P^{N, \b}_{\omega}$ corresponds to some measure on the set of domino tilings $\mathfrak D(\a, N,m)$. When $\beta=\frac{1}{2}$ we know from Proposition \ref{uniform} that the corresponding measure is the uniform measure, for general $\beta$ in can be shown in the same fashion that the corresponding measure is $\mathcal P_q$ for $q=\frac{\beta}{1-\beta}$ defined by

\begin{equation} \label{b_measure}
\mathcal P_{q}(D\in \mathfrak D(\a, N, m))=\frac{q^{\text{number of the horizontal dominos in }D}}{(1+q)^{N(N+1)/2}s_{\omega}(1^{N})}.
\end{equation}

This fact in the case of the Aztec diamond is well known, see \cite{J2} and \cite{BF}; for a recent
exposition see \cite{BCC}.

Next, using the very same arguments as in the case of  $\beta=\frac{1}{2}$ we can study the asymptotics of random domino tilings of $\mathcal R(\a, N,m)$ with respect to $\mathcal P_{q}.$
We decided to include only the statements of the results in this case.

\begin{prop}  \label{densityb} Consider $N\rightarrow \infty$ asymptotics such that all the
 dimensions of a rectangular Aztec diamond $\mathcal R(N, \a(N), m(N))$ linearly grow with
 $N$. Let $(\chi, \kappa)$ be the new continuous coordinates of the domain.
 Assume that the sequence of signatures $\omega(N)$ corresponding to the
 first row is regular and $ \lim\limits_{N\to\infty}
 m[\omega(N)]=\mes_\omega.$ Let us fix $\kappa \in (0, 1)$ and let
$\mes_q^\kappa$ be the limit of measures $m[\rho_{q}^k(N)],$ induced by $\mathcal P_{q}$  at the level $k.$
The density
of $\mes_q^\kappa$ can be computed in the following way
\begin{equation} \label{eq:density_formula1}
{\bold d} \mes_q^\kappa(x)=\frac{1}{\pi}\textup{Arg}( {\bold
  z}^\kappa_+(x)),
\end{equation}
where ${\bold z}^{\kappa, q}_+(x)$ is the unique complex root of the system
 \begin{equation} \label{sysb}
  \begin{dcases}
\F^\b_{\ka}(z,t)=\frac{z}{(1-\ka)}\Big (\frac{t}{z}- \frac{1}{(z-1)}+\frac{\kappa q}{1+q + q(z-1)} \Big ) + \frac{z}{z - 1}= x, \\
\textup{St}_{\mes_\omega}(t) =  \log(z).
  \end{dcases}
\end{equation}
which belongs to the upper half-plane. This
formula is valid for such $(x, \kappa)$ that the complex root
exists, the density is equal to zero or one otherwise.
\end{prop}

\begin{thm} \label{frozen_bb}
The frozen boundary of the limit of a random rectangular Aztec diamond of type
$(N, A^{(s)}, B^{(s)})$ with respect to the measure $\bold P^{N,q}_{\omega}$ is a rational algebraic curve $C_\b$ with an
explicit parametrization \eqref{eq:curveb}. Moreover, its
dual $C_q^{\vee}$ is of degree $2s$
and is given in the following parametric form
\begin{equation}\label{eq:curveb}
C_q^{\vee}=\left(\theta,
\text{ }  \frac {(1+q)\theta \Pi_s(\theta)}{ (\Pi_s(\theta)-1)(q\Pi_s(\theta)+1)} \right) ,
\end{equation}
where

$$\Pi_s(\theta)=\frac{(1-a_1\theta)(1-a_2 \theta)\cdots(1-a_s \theta)}{(1-b_1 \theta)(1-b_2
  \theta)\cdots(1-b_s \theta)}.$$
\end{thm}

The Central Limit Theorem \ref{theorem:covariance-domino} is also generalized straightforwardly, we omit its statement.

 \begin{figure}[h]
\includegraphics[width=0.5\linewidth]{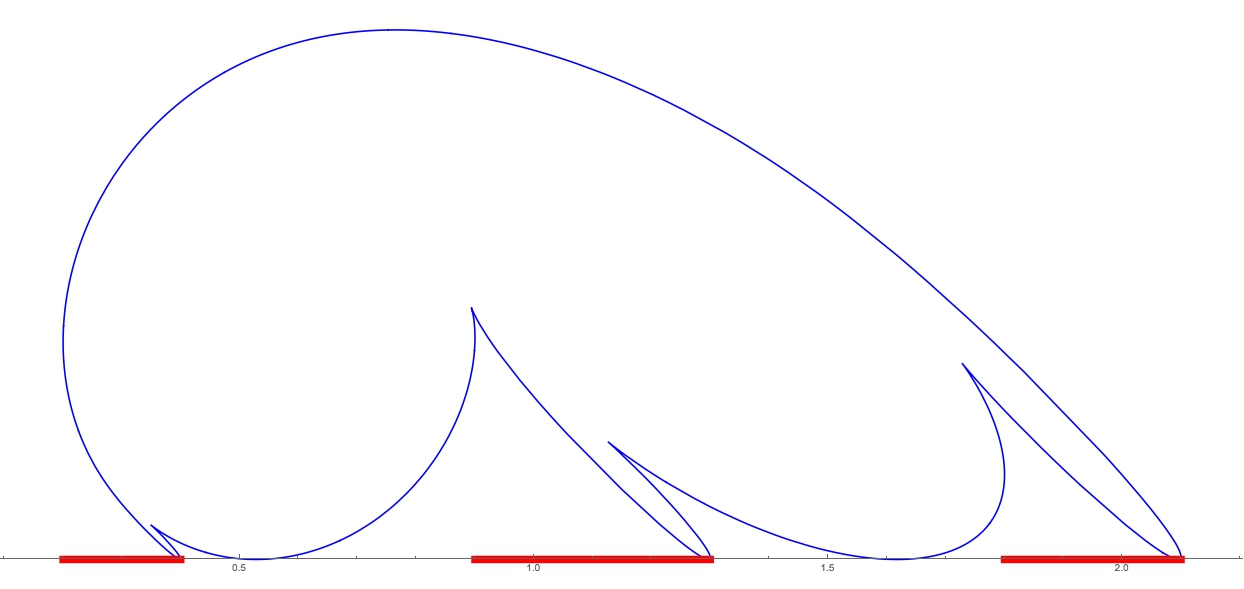}
  \label{fig:cloudb}
\caption{An example of a curve $C_\b$ with three boundary segments and $q=0.0099$.}
\end{figure}

 \begin{figure}[h]
\includegraphics[width=0.5\linewidth]{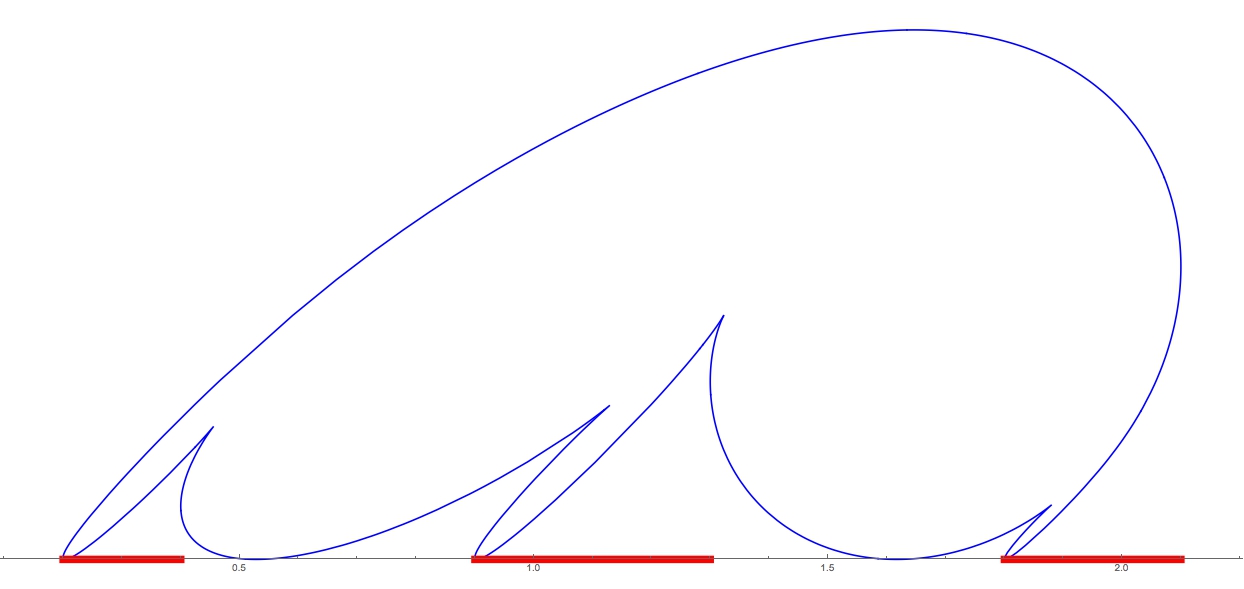}
  \label{fig:cloudb2}
\caption{An example of a curve $C_\b$ with three boundary segments and $q=99$.}
\end{figure}

%% file: AppendixB.tex
\section{Appendix B}
\label{sec:B}
In this section we show that the local fluctuations in our model are governed by the discrete sine kernel. The conjectured local limit first appeared in \cite{KOS}, and the proof for the Aztec diamond is given in \cite{CJY}. This is an immediate corollary of a result of \cite{G} about lozenge tilings, which uses the methods from \cite{P1}. Such a derivation is possible because the distribution of a signature on one level of our probabilistic model can be obtained from some random lozenge tilings model. However, it does not seem possible to obtain the two-dimensional local (or global) behavior with the help of this relation between domino and lozenge tilings.

First, let us formulate the general result from \cite{G}. Assume that $\{ \rho_N \}_{N \ge 1}$ is a sequence of probability measures which satisfies assumptions of Theorem \ref{theorem_moment_convergence}.
Let $0<a<1$, and let $\rho^a_N$ be the probability measure on $\GT_{[ a N]}$ with the Schur generating function 
\begin{equation} \label{eq:sch1}
\mathcal S^{U(N)}_{\rho^a_N} (u_1, \dots, u_{[a N]}) = \left. \mathcal S^{U(N)}_{\rho_N} (u_1, \dots, u_{N}) \right|_{u_{[aN]+1}=1, \dots, u_N=1},
\end{equation}
(such a measure exists because the coefficients in the branching rule for Schur functions are nonnegative).
In fact, due to \eqref{eq:sch1} the measure $\rho^a_N$ encodes the distribution of random lozenge tilings with random boundary conditions $\rho_N$ in the model considered by Petrov in \cite{P1}.
By Theorem \ref{theorem_moment_convergence}, the random probability measures $m[\rho^a_N]$ converge to a deterministic measure $\mes^a$. Let $\phi^{(a)} (x)$ be the density of $\mes^a$ with respect to the Lebesgue measure (recall that it always exist and takes values in $[0;1]$).

For $p \in (0;1)$ a \textit{discrete sine kernel} is defined by the formula
$$
K_p (y_1, y_2) = \frac{\sin (p \pi (y_1-y_2))}{\pi (y_1 - y_2)}, \qquad y_1, y_2 \in \mathbb Z.
$$

Let $\la = (\la_1 \ge \la_2 \ge \dots \ge \la_{[a N]})$ be a random signature distributed according to $\rho^a_N$. For $x_1, \dots, x_m \in \mathbb Z$ denote by $\theta^{(m)} (x_1, \dots, x_m)$ the probability that $\{ x_1, \dots, x_m \} \subset \{ \la_i +N-i \}_{i=1 \dots [aN]}$.
\begin{prop}
\label{prop:local-behavior}
Let $x \in \mathbb{R}$, and let $x(N)$ be a sequence of integers such that $x(N)/N \to x$, as $N \to \infty$. For $m \in \mathbb N$ let $x_1(N), \dots, x_m(N)$ be sequences of integers such that $x_i (N) - x(N)$ does not depend on $N$, $i=1, \dots, m$. 

Then, in the assumptions and notations above, we have
$$
\lim_{N \to \infty} \theta^{(m)} (x_1(N), \dots, x_m (N)) = \det_{i,j=1}^m \left[ K_{\phi^{(a)} (x)} (x_i(N), x_j(N)) \right],
$$
 $($note that the right-hand side does not depend on $N$ $)$.
\end{prop}
Proposition \ref{prop:local-behavior} immediately follows from Theorem 4.1 of \cite{G} and Theorem \ref{theorem_moment_convergence}.

Let us apply it to our setting. Let $k=k(N)$ be a sequence of integers such that $k(N)/N \to a$ as $N \to \infty$.
Recall that $\rho^{k(N)}$ is a probability measure on signatures of length $N-\lfloor\frac{k(N)-1}{2}\rfloor$ coming from uniform domino tilings, see Section \ref{sec:3}. Recall that the measures $m[\rho^{k(N)}]$ have a limit measure $\mes^{\kappa}$. The expression for its density ${\bold d} \mes^{\kappa}$ is given by Theorem \ref{density}.

Now let $\la = (\la_1 \ge \la_2 \ge \dots \ge \la_{N-\lfloor\frac{k-1}{2}\rfloor})$ be a random signature distributed according to $\rho^{k(N)}$. For $x_1, \dots, x_m \in \mathbb Z$ denote by $\tilde\theta^{(m)} (x_1, \dots, x_m)$ the probability that $\{ x_1, \dots, x_m \} \subset \{ \la_i +N-i \}_{i=1 \dots N-\lfloor\frac{k-1}{2}\rfloor}$.

\begin{prop}
Let $x \in \mathbb{R}$, and let $x(N)$ be a sequence of integers such that $x(N)/N \to x$, as $N \to \infty$. For $m \in \mathbb N$ let $x_1(N), \dots, x_m(N)$ be sequences of integers such that $x_i (N) - x(N)$ does not depend on $N$, $i=1, \dots, m$.

Then
$$
\lim_{N \to \infty} \tilde \theta^{(m)} (x_1(N), \dots, x_m(N)) = \det_{i,j=1}^m \left[ K_{{\bold d} \mes^{\kappa} (x)} (x_i(N), x_j(N)) \right],
$$
\end{prop}
 $($note that the right-hand side does not depend on $N$ $)$.
\begin{proof}
It directly follows from Lemma \ref{Schur_generating} and Proposition \ref{prop:local-behavior}.
\end{proof}
